\documentclass{elsarticle}
\usepackage{enumerate,amssymb,amsmath,slashbox}
\usepackage{graphicx}
\begin{document}

\newtheorem{conjecture}{Conjecture}
\newtheorem{convention}{Convention}
\newtheorem{definition}{Definition}

\newtheorem{theorem}{Theorem}

\newtheorem{lemma}[definition]{Lemma}
\newtheorem{proposition}[definition]{Proposition}
\newtheorem{question}{Question}
\newdefinition{remark}{Remark}
\newproof{proof}{Proof}

\def\A{\mathcal{A}}
\def\M{\mathcal{M}}
\def\P{\mathcal{P}}
\def\pdw#1{\mathrm{pdw}_{#1}}
\def\Q{\mathcal{Q}}
\def\Rset{\mathbb{R}}
\def\T{\mathcal{T}}

\def\blgh{i}
\def\alphaangle{ii}
\def\betaangle{iii}

\def\casenine{1}
\def\casetwofour{2}
\def\theothercases{3}

\title{
Spherical tilings by congruent quadrangles over pseudo-double
  wheels~({III}) | the essential uniqueness in case of convex tiles}
\date{\today}


\author[mi]{Yohji Akama\corref{cor1}}
\address[mi]{Mathematical Institute\\
  Graduate School of Science,
  Tohoku University\\
  Sendai, Miyagi 980-0845 JAPAN} 

\cortext[cor1]{Corresponding author}
\ead{akama@m.tohoku.ac.jp}
\ead[ur]{http://www.math.tohoku.ac.jp/akama/stcq/}

\author[mi]{Yudai Sakano}
\ead{my.sailingday.0827spc@gmail.com}

\begin{abstract} 
In [B.~Gr{\"u}nbaum, G.~C. Shephard, Spherical tilings with transitivity
  properties, in: The geometric vein, Springer, New York, 1981,
  pp. 65--98], they proved ``for every spherical \emph{normal} tiling by
  congruent tiles, if it is isohedral, then the graph is a Platonic
  solid, an Archimedean dual, an $n$-gonal bipyramid~($n\ge3$), or an
  $n$-gonal trapezohedron~(i.e., the pseudo-double wheel of $2n$
  faces).'' In the classification of spherical monohedral tilings, one
  naturally asks an ``inverse problem'' of their result: \emph{For a
  spherical monohedral tiling of the above mentioned topologies, when is
  the tiling isohedral?}  We prove that for any spherical monohedral
  quadrangular tiling being topologically a trapezohedron, if the number
  of faces is 6, or 8, if the tile is a kite, a dart or a rhombi, or if
  the tile is convex, then the tiling is isohedral.
\end{abstract}

\maketitle

\section{Introduction}\label{sec:intro}

In \cite{akama13:_class_of_spher_tilin_by_i}, we proved that
twelve copies of some spherical \emph{concave} quadrangle organize two
spherical tilings such that the two tilings have the same plain
graph~(the \emph{pseudo-double wheel} of twelve faces~\cite{MR2186681}),
but one spherical tiling~(See the right bottom of Figure~\ref{fig:chart_P_F}) is isohedral~(i.e. the symmetry group acts transitively on the tiles) while the other
not~(See Figure~\ref{chart:a}). The latter spherical
non-isohedral tiling by congruent \emph{concave} quadrangles over a
pseudo-double wheel is a counterexample of the inverse of
Gr\"unbaum-Shephard's result~\cite{MR661770} ``for every spherical
\emph{normal} tiling by congruent tiles, if it is isohedral, then the
graph is a Platonic solid, an Archimedean dual~\cite{MR2410150}, an $n$-gonal
bipyramid~($n\ge3$), or an $n$-gonal trapezohedron~(i.e., the
pseudo-double wheel of $2n$ faces).'' So we ask an ``inverse problem''
of their result: 
\begin{question}\label{q:akama}
For a
spherical monohedral tiling with the graph being a Platonic solid,
an Archimedean dual, an $n$-gonal bipyramid or an $n$-gonal
trapezohedron, when  is it isohedral?
\end{question}
According to \cite{sakano13:_class_of_spher_tilin_by_kdr}, when the tile
is a kite, a dart, or a rhombi, every spherical tiling over a
pseudo-double wheel is isohedral.  Moreover, by checking the number of
tiles and the \emph{vertex types} in the complete table of spherical
monohedral triangular tilings~\cite{MR1954054}, we can prove that
\begin{theorem}\label{thm:inverse_triangle}If a spherical tiling by
 congruent triangles is topologically a Platonic solid, an Archimedean
 dual, an $n$-gonal bipyramid, then 
 the tiling is isohedral.\end{theorem}

To state our partial answer of this inverse problem, recall that we cannot divide
the tile of a spherical monohedral quadrangular tiling into two
congruent triangles, if and only if the tile is either \emph{type~2} or
\emph{type~4} of
Figure~\ref{fig:type24_map}~(Figure~\ref{fig:type24_map}), since the tile
of a spherical monohedral quadrangular tiling has necessarily
equilateral adjacent edges~\cite[Proposition~1]{agaoka:quad}.
\begin{theorem}\label{thm:theo} Given a spherical  tiling consisting $\T$ of $F$ congruent
 quadrangles of type~2 or 4, such that
\begin{itemize}
\item
$F=6$ or $8$; or,  
\item the tile is  \emph{convex}
 and the graph of the tiling is $\pdw{F}$ with
 $F\ge10$. 
\end{itemize}
 Then $\T$ has  chart $\P_F$ as in Figure~\ref{fig:chart_P_F} or the mirror,
 and $\T$ is isohedral.
\begin{figure}[ht]
\begin{tabular}{c c}
\begin{minipage}{0.7\textwidth}
\includegraphics[width=10cm]{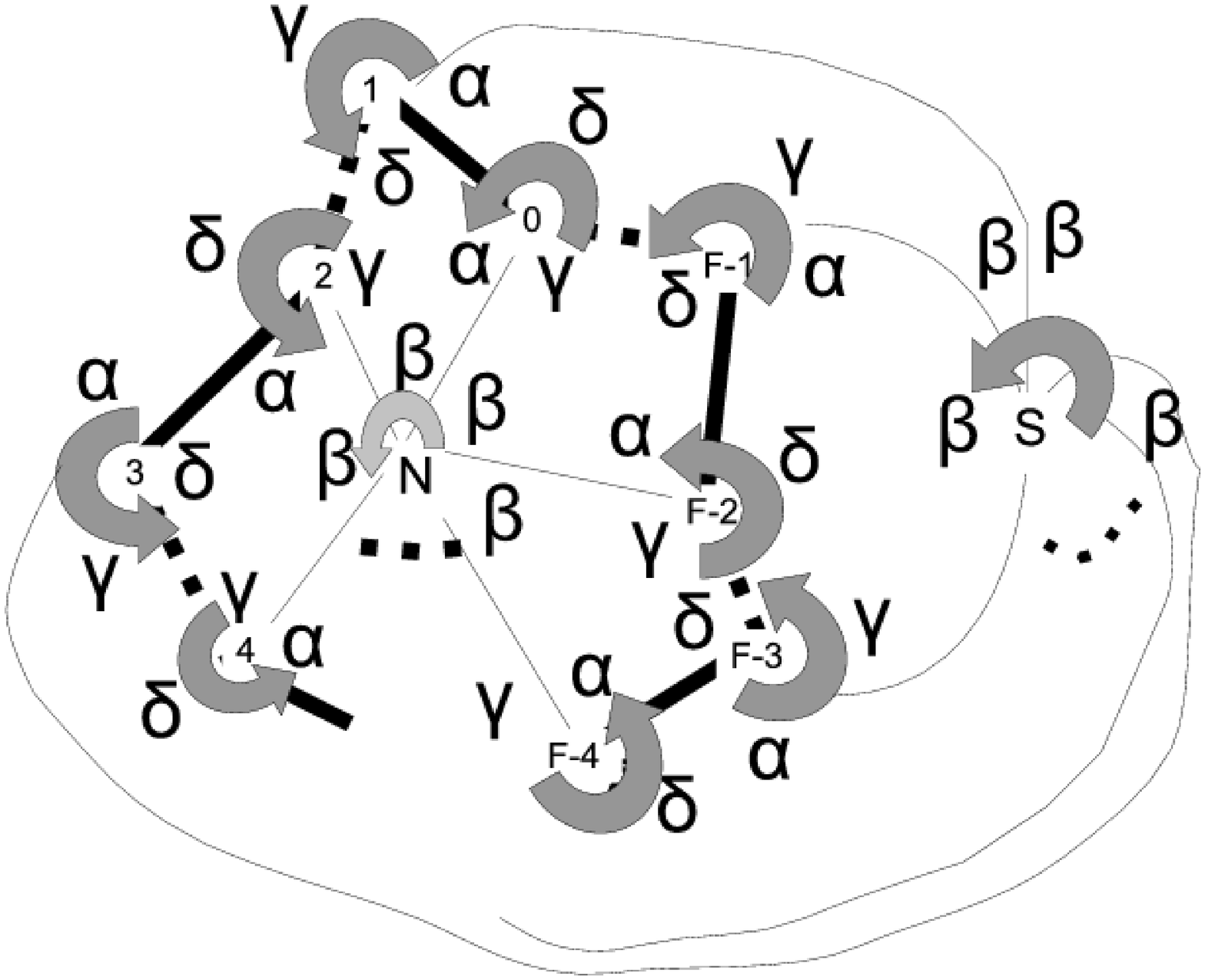}
\end{minipage}&
\begin{minipage}{0.3\textwidth}
\begin{tabular}{c}
\includegraphics[width=2cm]{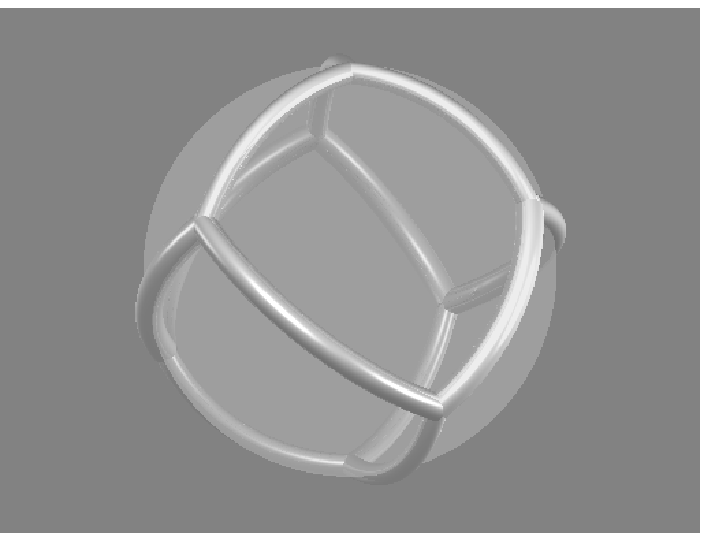}\\
\includegraphics[width=2cm]{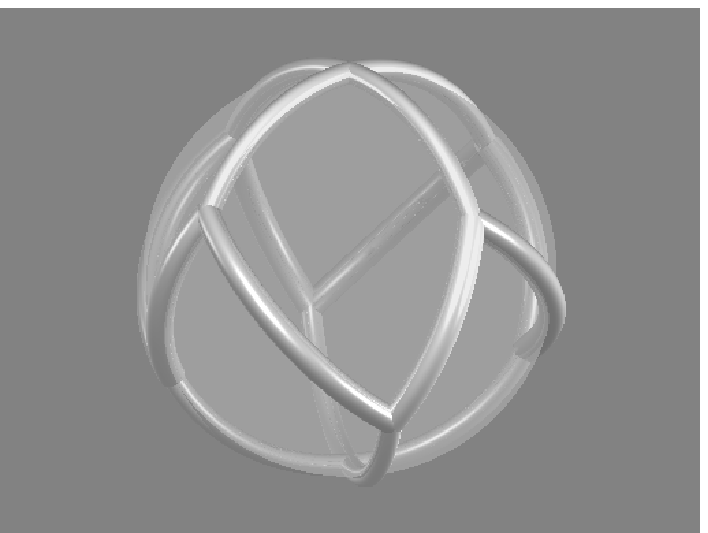}\\
\includegraphics[width=2cm]{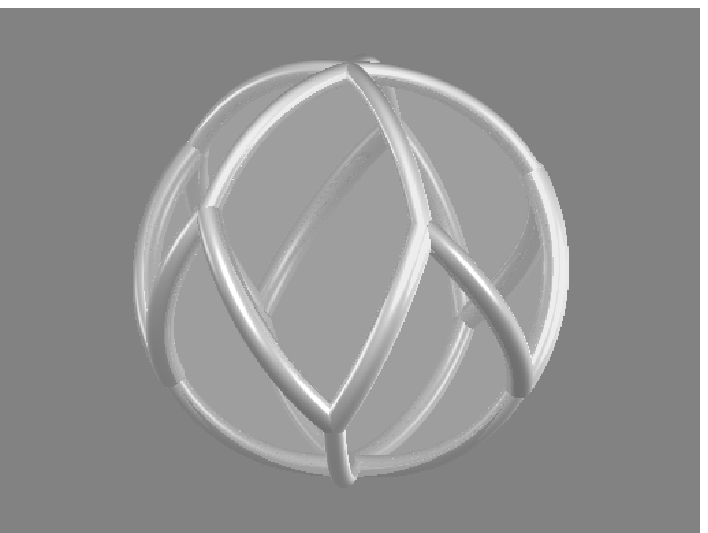}\\
\includegraphics[width=2cm]{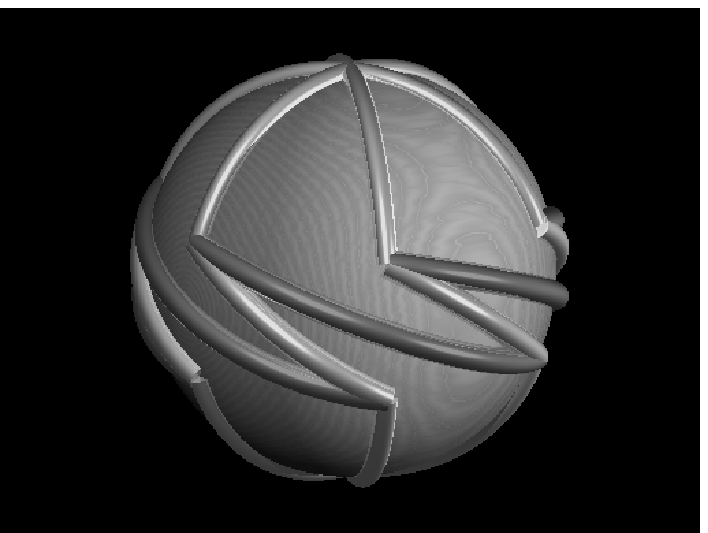}
\end{tabular}
\end{minipage}
\end{tabular}
\caption{ The chart $\P_F$ is defined by the left figure, whee the graph
 is a \emph{pseudo-double wheel of $F$ faces}~(Definition~\ref{def:pdw}). The whirl at each vertex
 indicates the cyclic order for the edges incident to the vertex.  The
 thick edges has length $b$, the dotted edges does length $c$, and the
 other edges does length $a$. When the tile is type~2, we have $a=c\ne
 b$, when the tile is type~4, the lengths $a,b,c$ are mutually distinct.
 The right column are examples of  spherical tilings by 6, 8, 10, 12
 congruent type-2
 quadrangles, from top to bottom. They are isohedral, because of
the axis through the pole and the center, and each axis through the midpoint
 of each non-meridian edge and the center. The twelve tiles of the last
 tiling organizes a non-isohedral tiling give in Figure~\ref{chart:a}.
 \label{fig:chart_P_F}}
\end{figure}
\end{theorem}
\begin{figure}[ht]
\begin{tabular}{c c}
\begin{minipage}{0.6\textwidth}
\includegraphics[width=7cm,height=7cm]{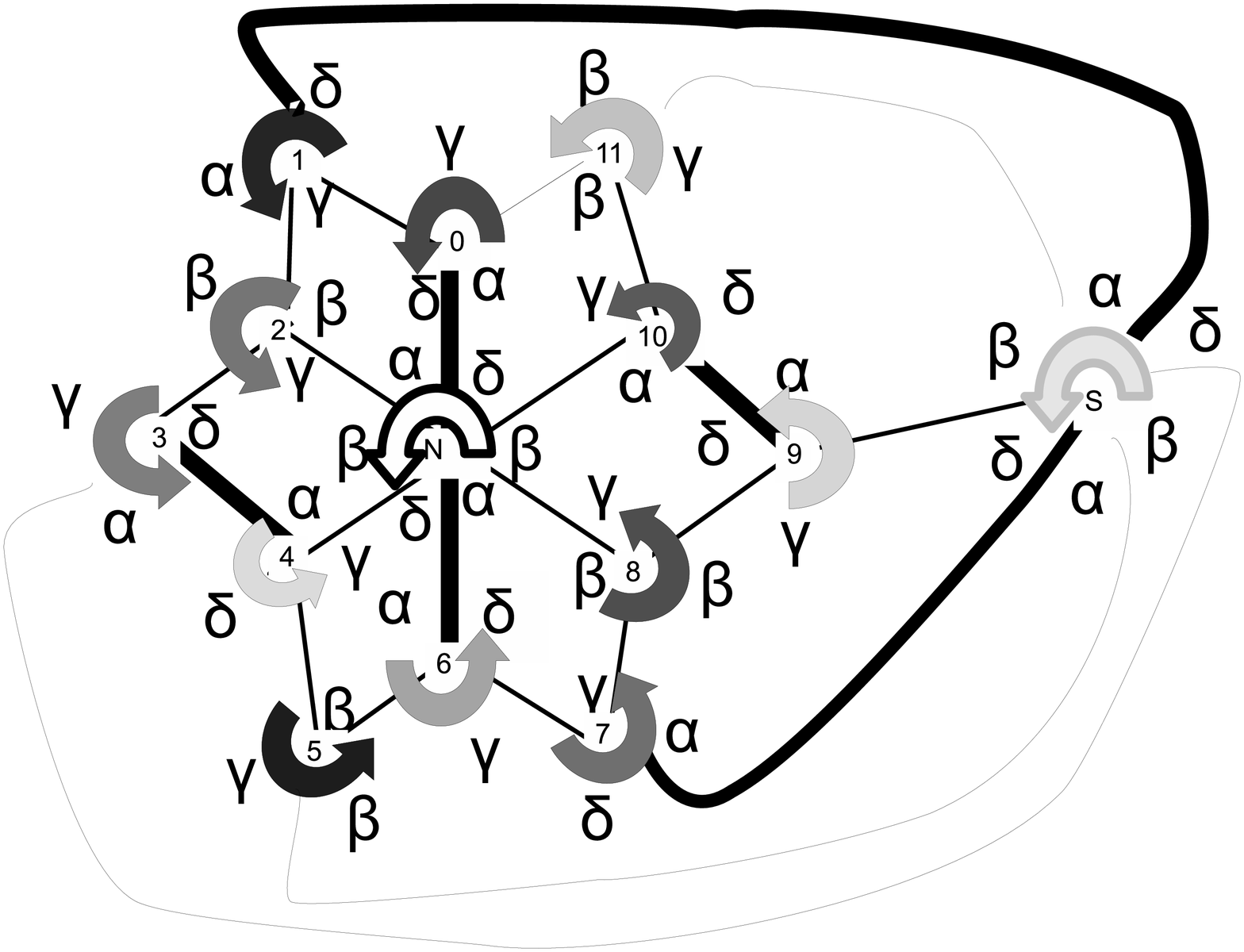}
\end{minipage}&
\begin{minipage}{0.3\textwidth}
\begin{tabular}{l}
\includegraphics[width=2.5cm,scale=0.25]{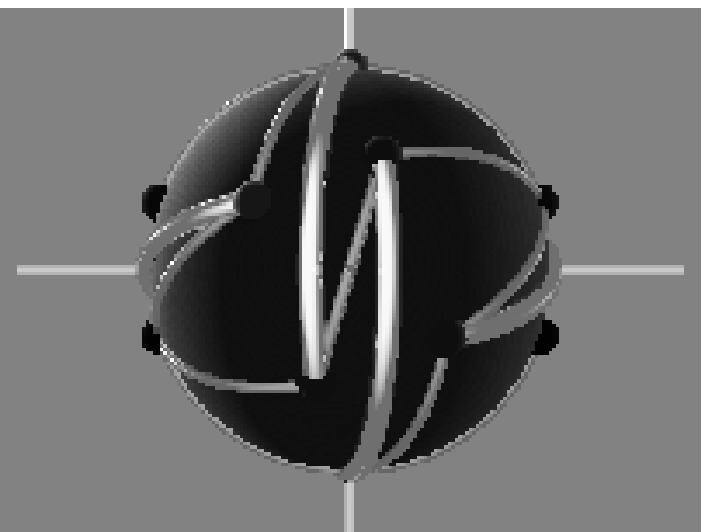}\\
\includegraphics[width=2.5cm,scale=0.25]{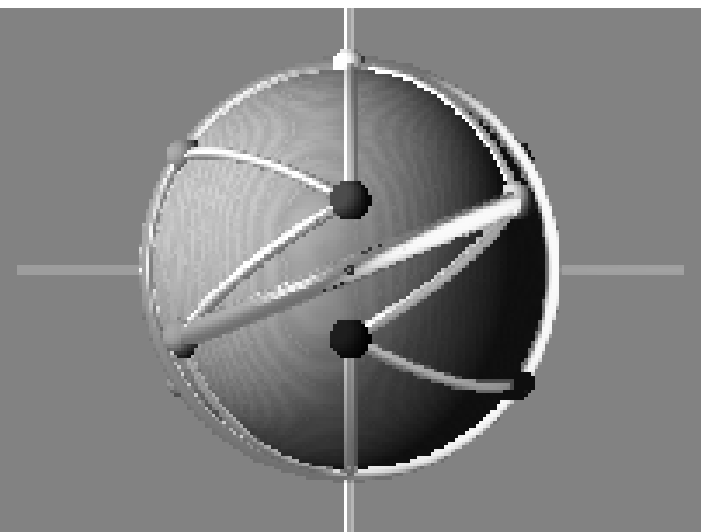}\\
\includegraphics[width=2.5cm,scale=0.25]{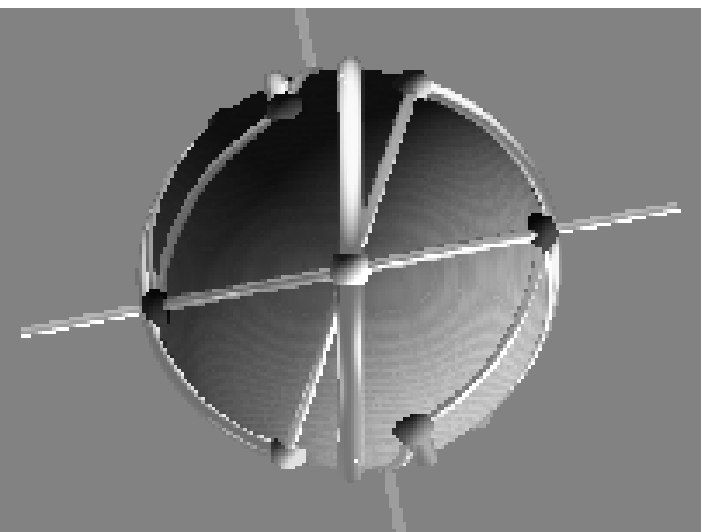}
\end{tabular}
\end{minipage}
\end{tabular}
\caption{The chart $\A$ of the spherical non-isohedral  tiling by twelve
 congruent concave quadrangles.  Thick~(resp. thin) edges of the
 chart correspond to edges of length $b$~(resp. $a$) of the tiling~(see
 Theorem~\ref{thm:k}). The  left upper figure is the view from a $2$-fold rotation axis through the midpoint between the vertices $v_0$
 and $v_1$. 
The  left middle  figure is from a $2$-fold rotation axis
 through the midpoint between the vertices $v_3$ and $v_4$.
The left bottom is from the other $2$-fold rotation axis
 through the poles. The twelve tiles of the 
 tiling organizes an isohedral tiling give in
 Figure~\ref{fig:chart_P_F}~(right bottom). \label{chart:a}}
\end{figure} 
\begin{figure}[ht]\centering
\includegraphics[width=6cm]{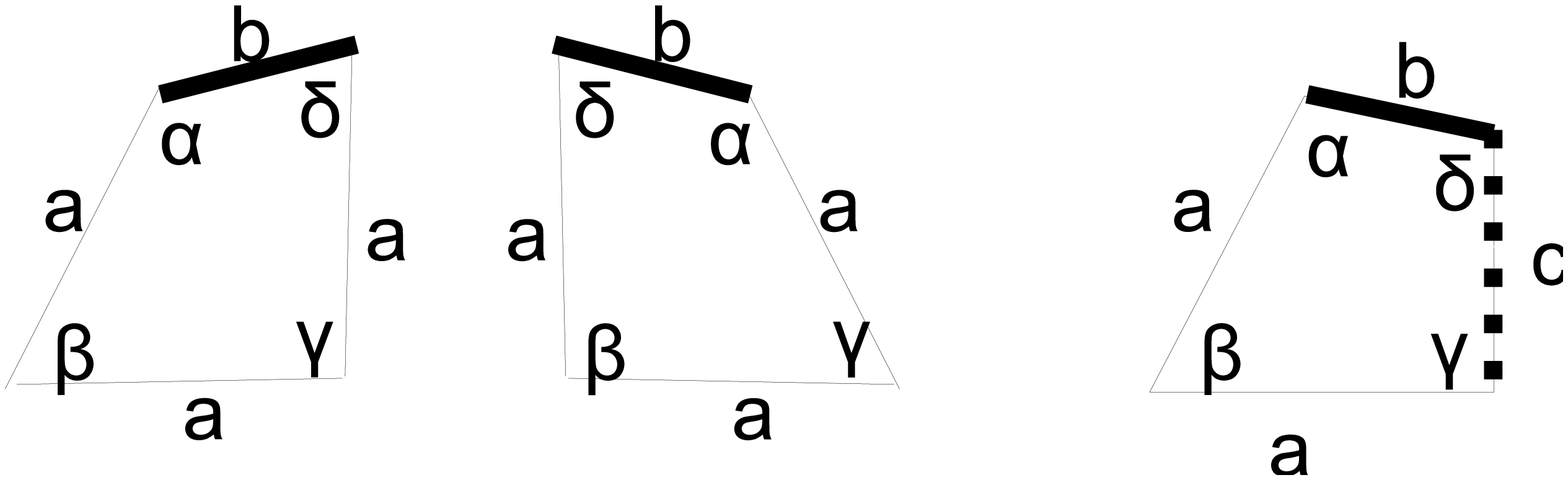}
\caption{Two quadrangles of type~2 and a quadrangle of type~4. The second
 quadrangle of type~2 is obtained from the first quadrangle of type~2 by
 swapping $(\alpha,\beta)$ and $(\delta,\gamma)$ with the length
 assignment unchanged. The chiralty of the first quadrangle of type~2
 is opposite to that of the second quadrangle of type~2.
 The variable $a,b,c$ are for edge-lengths and
 they have different 
 values.
 The variables $\alpha,\beta,\gamma,\delta$ are for angles.  \label{fig:type24_map}}
\end{figure}
In \cite{akama13:_spher_tilin_by_congr_quadr_ii}, we proved that such a tiling $\T$ of Theorem~\ref{thm:theo}
indeed exists.  Actually, we
proved that for every spherical tilings by congruent \emph{possibly
concave} quadrangles over pseudo-double wheels, we explicitly
represented every inner angles and every edge-length by the inner angle
$\gamma$~(and the inner angle $\alpha$, resp.) of the tile.

Our result is a first step toward the classification of the spherical
tilings by congruent quadrangles, because a \emph{quadrangulation of the
sphere} is mechanically obtained exactly from a \emph{pseudo-double
wheel}, by means of finite number of applications of two local
expansions of a graph~\cite{MR2186681}. See 
Proposition~\ref{prop:plantri}.

This paper is organized as follows: In the next section, we present a
setting to classify the spherical monohedral quadrangular tilings in
terms of planar graphs, angle- and length-assignments to the graph, and
matching theory.  In Section~\ref{sec:derivation}, we derive
Theorem~\ref{thm:theo} from Theorem~\ref{thm:Forbidden}. In
Section~\ref{sec:Forbidden}, we prove Theorem~\ref{thm:Forbidden}.  In
the final section, we conjecture that Question~\ref{q:akama} holds for
convex tiles, based on Theorem~\ref{thm:theo}. Then we propose to
classify a reasonable class of spherical monohedral quadrangular
tilings. Finally we show that the use of forbidden patters is a natural
approach to classify spherical tilings by congruent type-2 \emph{convex}
quadrangles, based on computer
experiment~\cite{akama13:_spher_tilin_by_congr_quadr}.

The authors thank Yoshio Agaoka and Takanobu Kamijo. The first
author thanks Gunnar Brinkmann,  Kris Coolsaet, and Nicolas van Cleemput 
 for discussion.

\section{A setting to classify the spherical monohedral quadrangular
 tilings}
\label{sec:setting}
We say a tiling by polygons is \emph{edge-to-edge}, if the vertices
and edges of tiles match. Unless otherwise stated, a tiling always implies an
edge-to-edge spherical tiling of the sphere, and is identified modulo
the special orthogonal group
$SO(3)$. We say two tilings are \emph{mirror image} to each other,
if they are different but identified module the orthogonal group $O(3)$.

\begin{definition}\label{def:map}
\begin{enumerate}
\item
A \emph{map} is $M=((V,E), \{A_v \}_{v\in V})$ such that
\begin{itemize}
\item $G=(V,E)$ is a nonoriented graph, where $V$ is a finite set of
		     vertices and $E$ is a finite set of edges. An edge
		     is a nonoriented pair of distinct vertices.

\item $A_v$ is the set of angles around $v$, that is, a set
$\{(v_1, v, v_2),  (v_2, v, v_3)$, \ldots, $(v_{n-1}, v, v_n), (v_n, v,
      v_1)\}$  
 such that $v_1, \ldots, v_n$ is the list of
       vertices adjacent to $v$. We write an angle
      $(u,v,w)$ by $\angle u v w$. 
\end{itemize}
\end{enumerate} 
\end{definition} 

\begin{definition}[Pseudo-double wheel~\protect{\cite{MR2186681}}]\label{def:pdw}
For an even number $F\ge 6$, a \emph{pseudo-double
wheel\/} $\pdw{F}$ with $F$ faces is a map such that

\begin{itemize}
\item the graph is obtained from a cycle
$(v_0, v_1,
v_2, \ldots, v_{F-1})$, by adjoining a new vertex $N$ to each $v_{2i}$
 $(0\le i<F/2)$ and then by adjoining a new vertex $S$ to each $v_{2i+1}$
 $(0\le i<F/2)$. We identify the suffix $i$ of the vertex $v_i$ modulo
 $F$. 
\item The inner angles at each vertex $v$ is defined naturally by the cyclic
      order at $v$.
The cyclic order 
at the vertex $N$ is defined as follows: the edge $N v_{2i+2}$ is next
      to the edge $N v_{2i}$. 
The cyclic order at the vertex $v_{2i}$ ($0\le i\le F/2$) is: the edge
 $v_{2i} N$ is next to the edge $v_{2i} v_{2i+1}$, which is next to the edge $v_{2i}
v_{2i-1}$.
The cyclic order 
at the vertex $S$ is: the edge $S v_{2i-1}$ is next to the edge $S v_{2i+1}$. 
The cyclic order at the vertex $v_{2i+1}$ ($0\le i< F/2$) is:
the edge $v_{2i+1} S$ is next to the edge $v_{2i+1} v_{2i}$, which is
      next to the edge $v_{2i+1}v_{2i+2}$.
\end{itemize}
\end{definition} 
We call each edge $N v_{2i}$ \emph{northern}, 
each edge $S v_{2i+1}$ \emph{southern}, and the other edges \emph{non-meridian}.
The number of edges is $2F$.

The form of a pseudo-double wheel is a graph consisting of the vertices,
edges and faces of Figure~\ref{fig:chart_P_F}. 

The graph of a spherical tiling by congruent quadrangles is a simple
quadrangulation of the sphere such that the minimum degree three. Here a
\emph{simple quadrangulation of the sphere} is a finite simple graph
embedded on the sphere such that every face is bounded by a walk of four
edges~\cite{MR2186681}. Let $Q_2$ be the class of simple
quadrangulations of the sphere such that the minimum degree three.  Let
$Q_3$ be the class of 3-connected simple quadrangulations of the
sphere. By the theorem of Steinitz, $Q_3$ is the class of
quadrangle-faced polytopal graphs.  A polytopal graph is the graph of
some polytope, where a polytope is, by definition, a bounded region
obtained by a finite number of half spaces. According to Table~2~(simple
quadrangulations with minimum degree three) and Table~3~(3-connected
quadrangulations) of \cite{MR2186681}, there is a simple spherical
quadrangulation $G$ of twelve quadrangular faces with the minimum degree
three but $G$ is not polytopal. By an \textsc{stcq} graph, we mean a
graph of some spherical tiling by congruent quadrangles. We do not know
whether every polytopal graph is an \textsc{stcq} graph or not. Anyway,
a quadrangulation of $Q_2$ is a candidate of an \textsc{stcq} graph.

\begin{proposition}[\protect{\cite[Theorem~2, Theorem~3]{MR2186681}}]\label{prop:plantri}Let
 $i$ be 2 or 3.
For every simple quadrangulation $G\in Q_i$, there is a sequence $G_0,
 G_1,\ldots, G_k=G$ of quadrangulations in $Q_i$  such that $G_0$
 is a
pseudo-double wheel and, for each $i$, $G_{i+1}$ can be obtained from
 $G_i$ by applying  two local
 expansions. Actually, all such are generated by a program
 \texttt{plantri}~\cite{plantri}. See Table~\ref{tbl:degrees} as for $Q_2$.
Specifically, 
the quadrangulations of $Q_i$ with  6 or 8 faces are exactly
 $\pdw{6}$ or $\pdw{8}$. 
\end{proposition}

\begin{definition}[Chart]\label{def:atlas}
\begin{enumerate}
\item A \emph{chart} of a tiling consists of following data:
\begin{itemize}
\item A map $M=((V,E), \{A_v\}_{v\in V})$.

\item A \emph{length-assignment} $L:E\to \Rset_{>0}$.

\item An \emph{angle-assignment} $K$, which is a function from
      $\bigcup_{v\in V} A_v$ to the set of affine combination of the
      variables $\alpha,\beta,\gamma,\delta$ over $\Rset$ subject to
\begin{align}
\sum_{a\in A_v}K(a) = 2\ [\pi\mathrm{rad}]\ \mbox{for each $v\in
 V$}.\label{eq:vertextypes}
\end{align}
\end{itemize} 

\item We say a chart is \emph{of type~$t$}, if each face is of type $t$.

\item The \emph{mirror image} of a chart $\M=((V,E), \{A_v\}_{v\in V}, L, K)$
      is a chart $\M^R=((V,E), \{A^R_v\}_{v\in V}, L, K^R)$ such that 
      $A^R_v:=\{\angle u v w\;;\; \angle w u v\in A_v\}$ and $K^R(\angle
      u v w)=K(\angle w v u)$. 

\end{enumerate}
\end{definition}

When we embed naturally the map $\pdw{F}$ into the sphere, the angles at
each vertex have
positive values with respect to the axial vector from the center to the
vertex.
\begin{definition}[Vertex types]\label{def:types}\label{def:XYtype}
The types of angles are, by definition, variables
 $\alpha,\beta,\gamma,\delta$. If the angles around a vertex $v$ are exactly
 $n_\alpha$ angles of type $\alpha$, $n_\beta$ angles of type $\beta$, 
$n_\gamma$ angles of type $\gamma$ and $n_\delta$ angles of $\delta$,
 then we say the \emph{vertex type} of $v$ is $n_\alpha \alpha + n_\beta \beta
+ n_\gamma \gamma+ n_\delta \delta$.
\end{definition}

\begin{lemma}\label{lem:area}
In a tiling by $F$ congruent quadrangles, every tile has area $4\pi/F$,
 which is the total of the inner angles subtracted by $2\pi$. In other
 words
\begin{align}
 \alpha+\beta+\gamma+\delta-2 = \frac{4}{F}. \label{eq:area}
\end{align}

 Hence there is no vertex of type $\alpha+\beta+\gamma+\delta$.
\end{lemma}

\begin{remark}\label{remark:mechanize}
Classification of spherical tilings by $F\ge10$ congruent quadrangles of
 type 2 or 4 is partially mechanizable.  For each even number $F\ge10$, generate all 2-connected
 simple quadrangulations $G$ of the sphere such that the minimum degree is
 three and the number of faces is $F$~(cf. Proposition~\ref{prop:plantri}).  Each tile of a spherical monohedral
 quadrangular tiling of type~2 or type~4 matches an adjacent tile at
 \emph{the} edge of length $b$. All such tile-matchings are generated
 mechanically, as the \emph{perfect matchings}~\cite{MR2744811} of the
 dual graph of the tiling's graph $G$~(If a graph with even number of faces
 are embeddable into an orientable, connected compact 2-fold, and all
 the faces are quadrangles, then the dual graph has a perfect
 matching~\cite{Carbonera06onthe}).
The extreme points of the edge of length $b$ have angle $\alpha$ and
$\delta$. By respecting this constraint, we can mechanically generate
all possible angle-assignments.  Every angle-assignment generates a
system of equations, that is, equation~\eqref{eq:area} and
equations~\eqref{eq:vertextypes}. If for every tile-matching, no such
systems of equations has a positive solution
$\alpha,\beta,\gamma,\delta<1$, then the spherical quadrangulation is not
realizable by a spherical tiling by $F$ congruent \emph{convex}
quadrangles of type~2 or 4.
\end{remark}

\section{Two forbidden
 (type-2/type-4) length-assignments and proof of Theorem~\ref{thm:theo}
 \label{sec:derivation}}

 Theorem~\ref{thm:theo} is derived from the following:
\begin{theorem}\label{thm:Forbidden}Given a spherical tiling $\T$ by $F$
 congruent quadrangles of type~2 or 4.
\begin{enumerate}\item \label{assert:forbidden_blade}
If the tile is convex or  $F=6, 8$, then the left pattern of Figure~\ref{fig:Forbidden} and
		       the mirror image are impossible for the
		       length-assignment of the tiling $\T$.

\item \label{assert:forbidden_typhoon}
If the tile is convex or $F=6, 8$, then
the right pattern in
 Figure~\ref{fig:Forbidden} and the mirror images 
are forbidden for the length-assignment of the tiling $\T$.
\end{enumerate}
\begin{figure}[ht]\centering
\includegraphics[width=12cm]{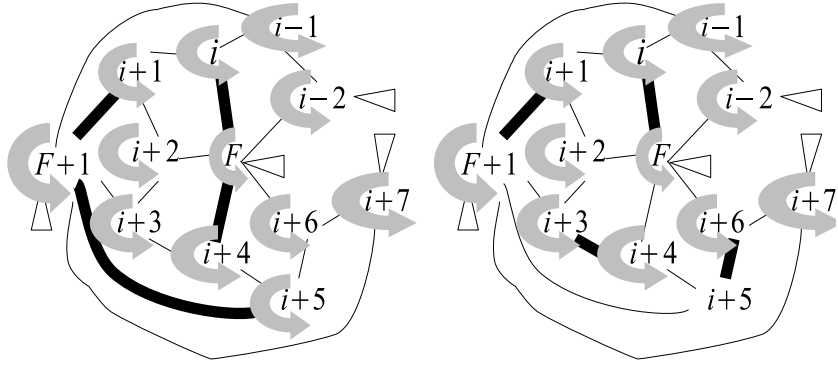} \caption{ For each $t=2,4$,
the two figures and their mirror images are impossible as the
length-assignment of a spherical tiling by congruent quadrangle of
type~$t$, if the tiles are \emph{convex}, or if the number of faces is $8$. Here thick edges are of length $b$ while
the others are of length $a$ or $c$.  A triangle indicates that
nonnegative number of edges \emph{may} occur at this position in the
cyclic order around the vertex (but they need not); When the number $F$
 of tiles is 8, for both patterns, the
pair of vertices designated by $i+6, i+7$ can be identical to the pair
of vertices designated by $i-2, i-1$. But in this case, all  triangle
 symbols are empty.
See Theorem~\ref{thm:Forbidden}.  \label{fig:Forbidden}}
\end{figure}
\end{theorem}
\begin{remark}\label{rem:automorphism}
By an \emph{automorphism} of a map $M$, we mean any
automorphism~\cite[Section~1.1]{MR2744811} $h$ of the graph that
preserves the cyclic orders of the vertices. Here we let $h$ send any
angle $\angle u v w$ to $\angle h(u) h(v) h(w)$.  For a chart
$\A=(M,L,K)$ and  an automorphism $h$ of the map $M$,  let $h(\A)$
be a chart $(M,\ L\circ h^{-1},\ K\circ h^{-1})$.

The correspondence of the vertices of Figure~\ref{fig:Forbidden}
\begin{align*}
 v_{i+k}\mapsto v_{i+5-k}\qquad(-2\le k\le 7) 
\end{align*}
preserves the cyclic orders of edges at each vertex, if the designated
 triangles in the figure
 are regarded as edges. Moreover the correspondence preserves the
 length-assignment of Figure~\ref{fig:Forbidden}~(left).
\end{remark}

As mentioned in the caption of Figure~\ref{fig:Forbidden}, the two patterns in Figure~\ref{fig:Forbidden} generalize the two
length-assignments~(Figure~\ref{fig:pdw8lghass}) for $\pdw{8}$ such that there
is a meridian edge of length $b$ and every face has only one edge of
length $b$. In the right figure, a meridian edge of length $b$ is
followed by two ``consecutive'' non-meridian edges of length $b$.  The
left figure contains two ``consecutive'' meridian edges of length $b$.
The two figures are not realizable by a spherical tiling by 8 congruent
quadrangles of type~2 or 4, according to
\cite{sakano11:_towar_class_of_spher_tilin}.

\begin{figure}[ht]
\centering
\includegraphics[width=5cm]{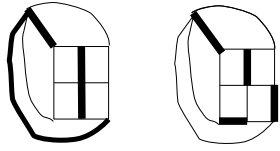}
\caption{The length-assignments of $\pdw{8}$  such that there is
 a meridian edge of length $b$ and every face has only one edge of
 length $b$. Here $8$ and $9$ are the north pole and the south pole.
\label{fig:pdw8lghass}}
\end{figure}

By applying Theorem~\ref{thm:Forbidden}
repeatedly to a slightly
 rotated map  around the axis through the north pole and the south
 pole, the only allowed length-assignment over $\pdw{F}$ satisfies
the assumption of
Theorem~\ref{thm:k}~(\cite[Theorem~2]{akama13:_class_of_spher_tilin_by_i}). That is,
diagrammatically speaking with Figure~\ref{chart:a}, the length $b$ edges
appear ``alternatingly'' in 
the meridian edges and the non-meridian edges, in the northern
hemisphere, and in the southern hemisphere. Here is an excerpt of \cite[Theorem~2]{akama13:_class_of_spher_tilin_by_i}:

\begin{theorem}\label{thm:k}
Assume a chart $\A$ satisfies the following  assumptions:
\begin{enumerate}[{$($}I{$)$}]\renewcommand{\theenumi}{\rm {\Roman{enumi}}}
\item \label{assumption:2} the map of the chart
     is $\pdw{F}$ for some $F\ge 10$, and

\item \label{assert:alternating} there is $b>0$ such that all the edges $N v_{6i}$, $v_{6i+1} S$
 and $v_{6i+3} v_{6i+4}$ have length $b$ for each nonnegative integer
 $i<F/6$ while the other edges do lengths $\ne b$.
\end{enumerate}
Then there  exists a  spherical \emph{non-isohedral} tiling $\T$ by congruent
 quadrangles, uniquely up to special orthogonal transformation.
 Moreover the tile is a \emph{concave} quadrangle of type~$2$. 
 The tiling $\T$ realizes $\A$, where
the length-assignment and the angle-assignment are
      Figure~$\ref{chart:a}$~$(F=12$ in particular\/$)$.

\renewcommand{\theenumi}{\arabic{enumi}}
\end{theorem}

By
Theorem~\ref{thm:k}~(\cite[Theorem~2]{akama13:_class_of_spher_tilin_by_i}),
the only allowed chart in question is exactly that of a spherical tiling by $F=12$
congruent \emph{concave} quadrangles of type~2. However, the tile is
convex, because of the assumption of Theorem~\ref{thm:theo}. Hence
\emph{every spherical tiling by $F\ge8$ congruent \emph{convex}
quadrangles of type~2 or 4 over $\pdw{F}$, every edge of length $b$ is
not meridian}. We can easily show the same holds for $F=6$, as in
\cite{sakano11:_towar_class_of_spher_tilin}.
Because we can prove the following in the rest of this section, 
 Theorem~\ref{thm:theo} follows from Theorem~\ref{thm:Forbidden}.

\begin{theorem}\label{thm:kouho2:one}Given a spherical tiling by $F\ge6$
 congruent possibly concave quadrangles of type~$t$ $(t\in \{2,4\})$ such that the map is $\pdw{F}$
 and all the edges of length $b$ are  non-meridian. Then, the chart modulo mirror
 image is  necessarily $\P_F$.
\end{theorem}
Various spherical tilings by congruent \emph{concave} quadrangles over
pseudo-double wheels are depicted in \cite{akama13:_spher_tilin_by_congr_quadr_ii}.

\def\N{v_F}
\def\S{v_{F+1}}
\medskip
\emph{From now on, the numbers designated at vertices in the figures
 and the symbols $v_i$'s
 have nothing to do with the vertices $v_i$'s introduced in the definition
 of pseudo-double wheels~(Definition~\ref{def:pdw}).}

\begin{convention}\label{conv}
Each displayed vertex is distinct from the others;
Edges that are completely drawn must occur in the cyclic order given in the picture;
Half-edges indicate that an edge must occur at this position in the cyclic order
around the vertex;
A triangle indicates that one or more edges may occur at this position in the cyclic
order around the vertex (but they need not);
If neither a half-edge nor a triangle is present in the angle between two edges in the
picture, then these two edges must follow each other directly in the cyclic order of
edges around that vertex.
\end{convention}

For any property $P(\alpha,\beta,\gamma,\delta)$ for the inner angles $\alpha,
\beta,\gamma,\delta$ of the tile, the \emph{conjugate property}
$P^*(\alpha,\beta,\gamma,\delta)$ is, by definition, a property
$P(\delta,\gamma,\beta,\alpha)$. 

\begin{figure}[ht]\centering
\includegraphics[width=6cm]{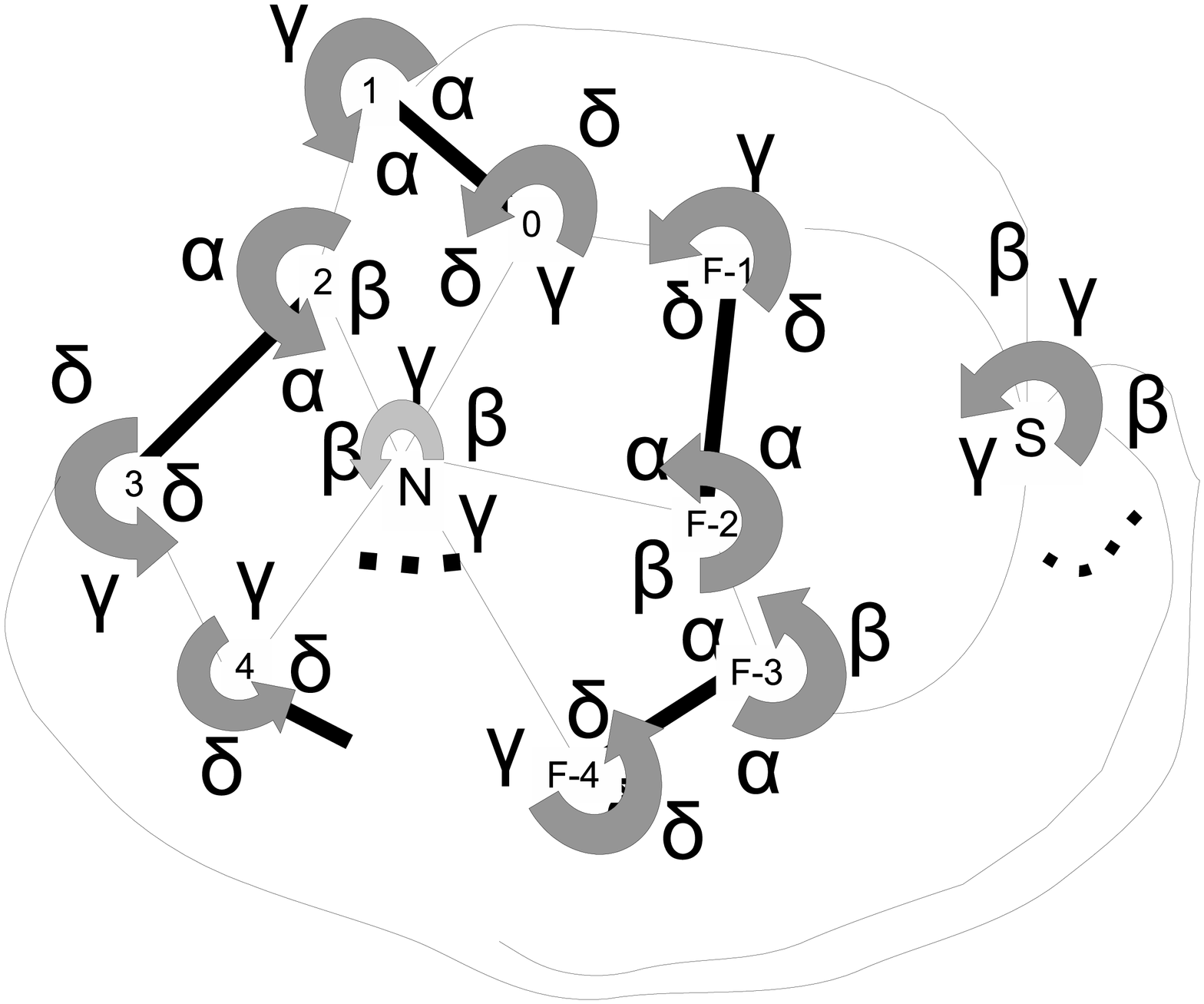}\includegraphics[width=6cm]{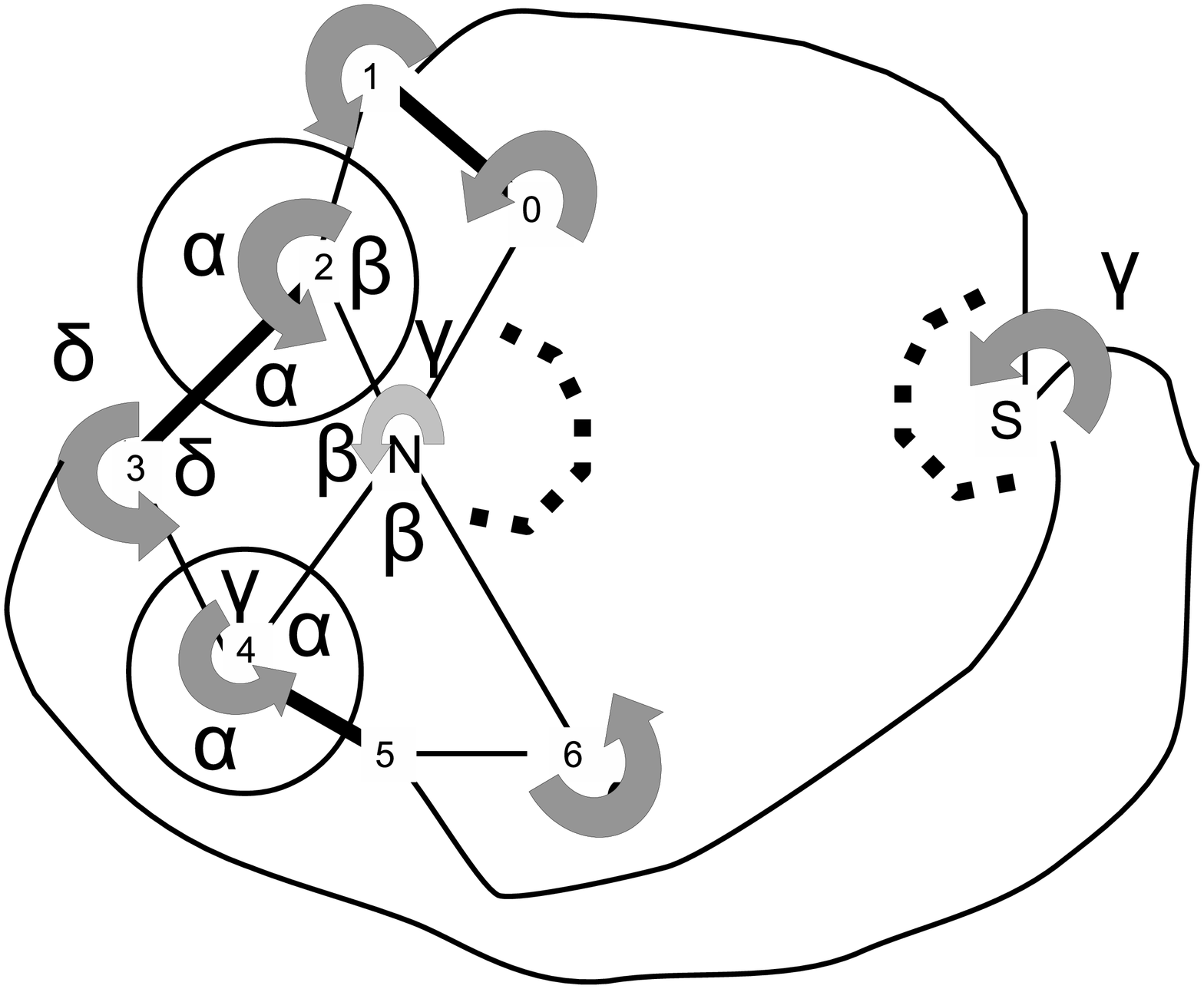}
\caption{The left is the chart $\Q_F$ of type~2. The
chart $\Q_F$ is obtained from the chart 
$\P_F$ by swapping the inner angles $\alpha\leftrightarrow
\delta$, $\beta\leftrightarrow\gamma$ in $F/2$ number of faces.
\label{fig:chart_Q_F}
The right is the chart of type~2 considered in Case~1 of the proof of Theorem~\ref{thm:kouho2:one} \label{fig:chart_QQ_F} }
\end{figure}

The rest of this section  is devoted to the proof of Theorem~\ref{thm:kouho2:one}.

Any non-meridian edge of length $b$ is not adjacent to
 another non-meridian edge of length $b$, because for any $t=2,4$, no
pair of edge of length $b$ are adjacent in a tile of type~$t$.

If the tile is of type~4, then all the edges of length $c$ are
 non-meridian. Otherwise, a meridian edge of length $c$ is adjacent
 to another non-meridian edge of length $b$ or to a meridian edge of
 length $b$. In the former case, the two edges of length $b$ are
 adjacent, which is impossible. The latter case is impossible as
we assumed all the edges of length $b$ are non-meridian. Thus we have
$\P_F$.

Hereafter we assume $t=2$. We have only to prove the case the dotted edges are of length
 $a$ in Figure~\ref{fig:chart_P_F}, because the other case is just a
 mirror image and is proved similarly.

Thus  the type of the pole
 $N$ is  $m \beta + n\gamma$ and  the
 type of the other pole $S$ is $i \beta + j \gamma$ for some $m,n,i,j$.
As the valences of the both poles are equal, we have $m+n=i+j=F/2$.
Then we have $m-i =
-(n - j)$.  

 If the two vertex types  are different, we
have $m-i\ne0$ or $n-j\ne0$.  Therefore
$m\beta+n\gamma = i \beta+ j\gamma$ implies
$\beta=\gamma=4/F$. Then the tile becomes a isosceles trapezoid, which
 implies $\alpha=\delta$. Thus by swapping $(\alpha,\beta)$ and
 $(\delta,\gamma)$ of suitable tiles, the resulting chart is
 $\P_F$, and is still equivalent to the original one.

We consider the other case where the two vertex types of the poles $N$
 and $S$ are equal, i.e., $m=i$ and $n=j$ .

If $m=0$, then we have (1).
If $n=0$, then by swapping inner angles
 $(\alpha,\beta)\leftrightarrow(\delta,\gamma)$ of all the tiles of type~2, we
 have $\P_F$. 

 Consider the case where there is a vertex of type
 $\alpha+\gamma+\delta$, or a vertex of type $\alpha+\beta+\delta$. Then
 Lemma~\ref{lem:area} implies $\beta=4/F$ or $\gamma=4/F$. In the
 former case, we have
 $m \frac{4}{F} + n\gamma=2$. As $m+n=F/2$, we have $4n/F - n\gamma=0$
 and so $\gamma=\beta=4/F$.   In the latter case $\gamma=4/F$, we
 have $\beta=\gamma=4/F$ similarly. Hence the quadrangle of type~2 is an isosceles
 trapezoid, and thus
 $\alpha=\delta$. So in all tiles  contributing the angle
 $\gamma$ to the vertex type of the pole $N$ or $S$,  if  we swap
 $(\alpha,\beta)\leftrightarrow(\delta,\beta)$ in a chart, then the resulting tiling is congruent to the
 original tiling. Therefore, the tiling has chart $\P_F$ of
 type~2 or the mirror image.

Next,  consider the other case, that is, the case where $m n \ne 0$ and
 any vertex
 type is neither $\alpha+\beta+\delta$ nor
 $\alpha+\gamma+\delta$. We have three cases.

Case~1. Some angle $\beta$ is adjacent to another angle $\beta$ at the
 pole $N$;

Case~1$^*$. 
Some angle $\gamma$ is adjacent to another angle $\gamma$ at  the pole $N$; or

Case~2. Angles $\beta$ and angles $\gamma$ appear alternatingly around
     the pole $N$.

\noindent
In all cases, we fix the location of edges of length $b$:
\begin{align}
\mbox{an edge}\ v_{2i} v_{2i+1} \ \mbox{has length $b$} \quad \left(0\le i\le (F-2)/2 \right) \label{gg4}.
\end{align} 
Once the conclusion of Theorem~\ref{thm:kouho2:one} is established in Case~1, the same
 argument with $(\alpha,\beta)$ and
 $(\delta,\gamma)$ swapped establishes the conclusion also in Case~1$^*$. 

Without loss of generality, we can assume, as in Figure~\ref{fig:chart_Q_F}~(left),
$\angle v_{2i} N v_{2i+2}$ is $\gamma\ (i=0);\ \beta\ (i=1)$, both in
Case~1 and Case~2.
By this and the property \eqref{gg4} with $i=0$, we have
$\angle N v_2 v_{2i+1}=\beta\ (i=0);\ \alpha\ (i=1)$.
By~\eqref{gg4},
\begin{align}\angle v_{2j-1} v_{2j} v_{2j+1}=\alpha,\ \mbox{or}\ 
\delta \label{gg8}.
\end{align}
Because no vertex has type $\alpha+\beta+\delta$, 
$\angle v_3 v_2 v_1 = \alpha$.
Hence
\begin{align}
 \angle v_3 S v_1 = \gamma. \label{gg9}
\end{align}
So, by $\angle v_3 v_2 v_1 = \alpha$ and by $  \angle N v_2
v_{2i+1}=\beta\ (i=0);\ \alpha\ (i=1)$, we have
\begin{align}
 2\alpha + \beta  = 2 . \label{gg10}
\end{align}

On the other hand, by \eqref{gg4} with $i=1$, we have
\begin{align}
 \angle N v_4 v_3 = \gamma. \label{gg11} 
\end{align} 

Now consider Case~1~(See Figure~\ref{fig:chart_QQ_F}~(right)).  Without
 loss of generality, we can assume $\angle v_4 N v_6=\beta$. Then the
 property \eqref{gg4} with $i=2$ implies $\angle N v_4 v_5 = \alpha$.
 Because of \eqref{gg8} and the absence of the vertex type $\alpha
 +\gamma + \delta$, we have $ \angle v_3 v_4 v_5 = \alpha$.  By this,
 \eqref{gg11} and $\angle N v_4 v_5 = \alpha$, we have $2\alpha +
 \gamma=2$. By this and \eqref{gg10}, the angle $\beta$ is equal to
 $\gamma$. Thus, as above, the chart is $\P_F$ or the mirror image.

In Case~2, the number $F$ of tiles is a multiple of four and is greater
than or equal to 8. Moreover angle $\beta$ occurs in the type of vertex
$N$ same time as angle $\gamma$.  So, without loss of generality, we
assume for each positive integer $k \le F/4$, it holds that $\angle
v_{4k+2 } N v_{4k+4 } = \beta$ and
\begin{align}
\angle v_{4k} N v_{4k+2} = \gamma. \label{gg14} \end{align}
As with  \eqref{gg9}, we can derive \eqref{gg11} and
$ \angle v_{4k+1  } S  v_{4k+3  } = \gamma$.
Since
$ \angle v_{4k+3 } S  v_{4k+5}$ is
$\beta $,  the opposite angle is
$ \angle v_{4k+3   } v_{4k+4  } v_{4k+5 }  =
 \delta$. 
By \eqref{gg4} and \eqref{gg14}, we have $\angle N v_{4k} v_{4k+1} = \delta$. 
This chart is  $\Q_F$ of type~2, as in
 Figure~\ref{fig:chart_Q_F}~(left).  

 The chart $\Q_F$ is not realizable by some spherical tiling by congruent \emph{concave} quadrangles of
type~2. Assume $\Q_F$ is. Then
the edge length $a$ is greater than
 $\pi/2$. Otherwise, the vertices $v_{2i}$'s are on the northern
 hemisphere and the vertices $v_{2i-1}$ on the southern, from which
 $\alpha,\delta<1$ follows. Hence, as 
 $\beta,\gamma<1$ by $F\ge8$,  the tile is convex on the contrary to the premise. So $a>\pi/2$. Thus,
 the tile $N v_2 v_3 v_4$ implies
  $\delta>1$ but the tile $N v_0 v_1 v_2$ does $\delta<1$, which is a contradiction. 

 No spherical tiling by congruent \emph{convex} quadrangles of
type~2 realizes the chart $\Q_F$ of
Figure~\ref{fig:chart_Q_F}~(left).
Assume otherwise. Apply  the following lemma:
\begin{lemma}[\protect{\cite[Lemma~2]{akama13:_spher_tilin_by_congr_quadr_ii}}]\label{lem:convex}
If there is a quadrangle $A B C D$ of type~2 such that the vertex
$B$ is located on a pole while the length between the opposite vertex $D$ and the other pole
is equal to the lengths of edges $BC$ and $CD$, then
$\angle B C D + \angle D B C $ is the angle of line. If the quadrangle $A B C D$ is convex, then 
$\angle D B A< \angle C B A$. 
\end{lemma} 

\begin{proof} Consider the vertex $\overline{C}$ antipodal to $C$. Then,
$CB\overline{C}D$ is a lune, so $\angle B C D=\angle
B\overline{C}D$. Because
$DB=B\overline{C}=D\overline{C}$,  we have a regular triangle, so
 $\angle
\overline{C}DB=\angle B \overline{C} D$.
 Hence $\angle B C D + \angle D B C =\angle  \overline{C} D B + \angle
D B C  = \angle \overline{C} D C$ is the angle of line. \qed\end{proof} 

Then by Lemma~\ref{lem:convex}, $\angle v_{F-1} N v_0=\pi -
\gamma\pi$ because the tile is of type~2 while $\angle v_{F-1} N v_0 <
\beta\pi $ because the tile is convex. Therefore, 
$\beta\pi+\gamma\pi>\angle v_{F-1} N v_0 +\gamma\pi=\pi$. The vertex $N$ has a vertex type
$(F/4)(\beta+\gamma)$ greater than $2$. This is a contradiction. 

Hence, whether the tile is convex or concave,
no spherical tiling by
 congruent quadrangles of type~2 realizes the chart $\Q_F$.
Thus the chart should be $\P_F$.
This completes the proof of Theorem~\ref{thm:kouho2:one}.

Hence the derivation of Theorem~\ref{thm:theo} from Theorem~\ref{thm:Forbidden}
is complete.

\section{Proof of Theorem~\ref{thm:Forbidden}\label{sec:Forbidden}}

The following  Subsection~\ref{subsec:former} and
 Subsection~\ref{subsec:latter} respectively prove
Theorem~\ref{thm:Forbidden}~\eqref{assert:forbidden_blade} and \eqref{assert:forbidden_typhoon}.
For each subsection, we will prove a lemma that generates all candidates
of angle-assignments on the length-assignment, and then will reject them.

We try to isolate the tile's convexity assumption from the proof of
Theorem~\ref{thm:Forbidden}, because by knowing how concave tile can
escape from the two forbidden patterns, there might be a chance to find
a new spherical tiling by concave quadrangles over a pseudo-double
wheel.

We use the following lemmas. A quadrangle of type~2 or type~4 satisfies
the following disjunctions on inner angle's equality.

\begin{lemma} \label{lune}
\begin{enumerate}
\item \label{assert:lune:1}
In a quadrangle of type~2 or type~4, $\alpha\ne\beta$ or $\gamma\ne\delta$.
\item\label{assert:lune:2}
In a quadrangle of type~4, $\alpha\ne\delta$ or $\beta\ne\gamma$.
\end{enumerate}
\end{lemma}

\begin{proof}\eqref{assert:lune:1} Assume otherwise. Consider a lune
of which boundary
 contains the two edge $\alpha\delta$ and $\beta\gamma$ of the quadrangle.
 Then we have an isosceles triangle beside the edge $\alpha\beta$
 of the tile because $\alpha=\beta$, and we have an isosceles triangle
 beside the edge $ \gamma\delta$ of the tile because $\gamma=\delta$.
Because the two sides of any lune has equal length, the length $b$ of the edge
 $\alpha\delta$ should be equal to the length $a$ of the edge
 $\beta\gamma$. \eqref{assert:lune:2} is proved similarly.\qed
\end{proof}

\begin{lemma}[\protect{\cite[Lemma~2]{akama13:_class_of_spher_tilin_by_i}}] \label{lem:6.2}In a quadrangle of
type~2, $\beta\ne\delta$ and $\alpha\ne\gamma$. In a quadrangle of type~4,
$\alpha\ne\gamma$. 
\end{lemma}
\begin{lemma}[\protect{\cite[Lemma~5]{akama13:_class_of_spher_tilin_by_i}}]\label{lem:u}
Suppose a spherical tiling by  congruent quadrangles of type~2 has (i) a 3-valent vertex $v$
     incident to 3 edges  of length $a$, and (ii) a 3-valent vertex
     incident to an edge of length $a$ and to an edge of length
     $b$. Then
 $\alpha\ne\delta$ and $\beta\ne\gamma$. 
\end{lemma}

\subsection{The proof of Theorem~\ref{thm:Forbidden}~\eqref{assert:forbidden_blade}
\label{subsec:former}}

By determining the vertex type of the vertices $F$ and $F+1$ in the
following Lemma~\ref{cor:homo}, we generate ten cases by fact  $h$ given
in Remark~\ref{rem:automorphism} is an automorphism of the map Figure~\ref{fig:Forbidden}~(left).
Then we will reject all the cases in Lemma~\ref{lem:blade_impossible}.

\begin{definition}\label{def:AB}In Figure~\ref{fig:Forbidden}~(left), 
let $A$ be the sum of the types of the four
designated inner angles  $\angle v_{i-2} \N
v_i$, $\angle v_i \N v_{i+2}$, $\angle v_{i+2} \N v_{i+4}$, and $\angle v_{i+4}
\N v_{i+6}$ around the vertex $F$, and let $B$ be  the sum of the types
 of  the four designated inner angles $\angle v_{i-1} \S v_{i+1}$, $\angle v_{i+1} \S
v_{i+3}$, $\angle v_{i+3} \S v_{i+5}$ and $\angle v_{i+5} \S
v_{i+7}$ around
the vertex $F+1$.\end{definition}

\begin{lemma}\label{lem:2a2d}If $F=8$ or the tile is convex, then
$2\alpha+2\delta \not\in\{A, B\}$.
\end{lemma}\begin{proof}
Otherwise by considering the vertex types of the vertices $F$ and $F+1$,
we have $2\alpha+2\delta\le2$~(the equality holds when the number of
faces is $8$), and thus $2\alpha\le 1$ or $2\delta\le1$. Here the vertex
types of $i$ and $i+4$ is $(\alpha+X+\delta, \alpha+X'+\delta)$,
$(2\alpha+X, X'+2\delta)$ or $(X+2\delta, 2\alpha+X')$ for some $X,
X'\in \{\beta,\gamma\}$.  Thus we have $ \beta\ge1$ or $\gamma\ge1$.
But the tiling in question is
 edge-to-edge when $F=8$ or else against the convexity of the tile.\qed
\end{proof}

\begin{lemma}\label{cor:homo} Suppose
the chart is of the form Figure~\ref{fig:Forbidden}~(left). Then if
 $2\alpha+2\delta\not\in\{A,B\}$, then
$A=B=\alpha+3\delta$ or $3\alpha+\delta$ and the
tile is a quadrangle of type~2.
 In this case, the list of such patterns modulo the automorphism of
 Remark~\ref{rem:automorphism} consists of ten patterns in Figure~\ref{fig:left_10possibilities}
 and those with $\alpha$ and $\delta$ swapped.

\begin{figure}[ht]
\centering
\includegraphics[width=12cm]{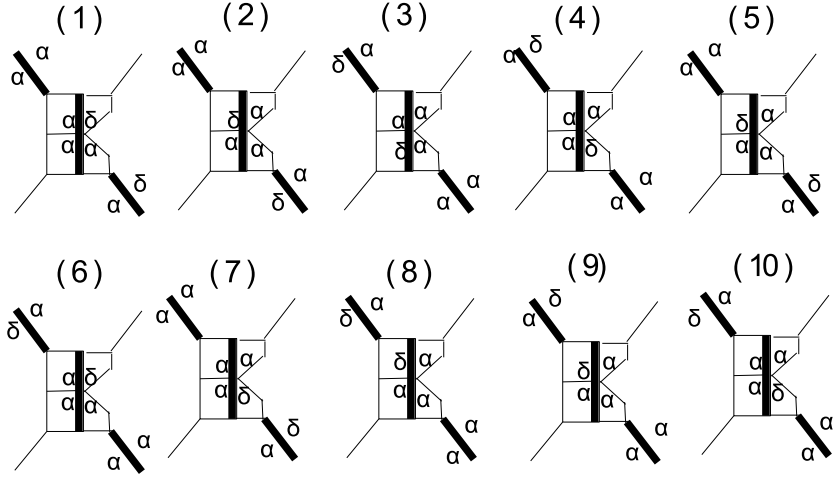}
\caption{The 
ten patterns in this figure and those with $\alpha$ and $\delta$ swapped forms the complete list of charts of the form
 Figure~\ref{fig:Forbidden}~(left) satisfying $F=8$ or the convexity of
 the tile. The central point~(infinite point) is
 $v_F$~($v_{F+1}$) of Lemma~\ref{cor:homo}.
We will prove all of them are forbidden in
 Subsection~\ref{subsec:former}. In the first four cases, the eight designated inner angles is
 unchanged by the automorphism given in Remark~\ref{rem:automorphism}.
 \label{fig:left_10possibilities}}
\end{figure}
\end{lemma}

\begin{proof} We claim
\begin{align}A\not\in\{4\alpha, 4\delta\}.\label{4a4d}\end{align} 
Consider the case where the tile is a quadrangle of type~4. If $A=4\alpha$, then
 the vertex $i$ is incident to two edges of length $c$, which is a
 contradiction. If $A=4\delta$, then the edge $v_{i+2} \N $ has
 length $c$. Therefore the edge $\S v_{i+3}$ does so and the angle
 $\angle v_{i+1} v_{i+2} v_{i+3}$ as well as the angle $\angle v_{i+3} v_{i+4}
 v_{i+5}$ 
has type $\beta$. As the edge
 $v_{i+5} \S $ and the edge $v_{i+1}
\S$  have length $b$ and $A=4\delta$, the angle $\angle
 \N  v_{i+2} v_{i+1}$ and the angle $\angle v_{i+3} v_{i+2}
\N$ have type $\gamma$. By comparing the type $\beta+2\gamma$ of vertices $i+2$
 and that $2\alpha+\beta$ of $i+4$, we have $\alpha=\gamma$, which contradicts against
 Lemma~~\ref{lem:6.2}.  

Consider the case where the tile is a quadrangle of type~2. Assume $A=4\alpha$.
Then the types of the vertex ${i+2}$ and the vertex ${i+3}$ are both
\begin{align}
2\beta+\gamma.\label{2bg}
\end{align}
Otherwise the type of vertex $i+2$ is $3\beta$ and thus that of vertex
 $i+3$ is $2\gamma+Y$ for some $Y\in \{\beta,\gamma\}$. Then
 $\beta=\gamma=2/3$. As the vertex $i$ is 3-valent and incident to a
 length-$a$ edge and a length-$b$ edge, we have a contradiction against
 Lemma~\ref{lem:u}.

 Hence the vertex $i+4$ has type $2\delta+\gamma$. By \eqref{2bg}, we
 have $\beta=\delta$, which contradicts against Lemma~\ref{lem:6.2}. We
 can similarly prove $A\ne4\delta$ for the tile being a
 quadrangle of type~2. This completes the proof of \eqref{4a4d}.
\medskip

We can prove similarly $B\not\in \{4\alpha, 4\delta\}$ by using the
 automorphism $h$ given in Remark~\ref{rem:automorphism}. So, by Lemma~\eqref{lem:2a2d}
 and \eqref{4a4d},
\begin{align}\label{claim:aho}
\{ A, B\}\subseteq\{3\alpha+\delta,\alpha+3\delta\}. 
\end{align}

We prove that the tile is not a quadrangle of type~4. Assume otherwise.
Then none of $A$ and $B$ is $3\alpha+\delta$, because two occurrences of
 inner angle
$\alpha$ necessarily share an edge of length $b$, so the other end-point
is incident to two edges of length $c$, a contradiction. Moreover we
have $A\ne\alpha+3\delta$. Assume otherwise. When $\angle v_i \N
v_{i+2}\ne\delta$, the angle is $\alpha$. By $A=\alpha+3\delta$, we have
$\angle v_{i+4} \N v_{i+2} = \delta$ and $v_{i+2} \N=c$, which is a
contradiction. Therefore $\angle v_i \N v_{i+2}=\delta$. Thus $\N
v_{i+2}=c=\S v_{i+3} = \S v_{i-1}$. It follows that the vertex $i+2$ has
type $\beta+2\gamma=2$. On the other hand, the assumption
$A=\alpha+3\delta$ implies $\angle v_i \N v_{i-2} = \delta $ or $\angle
v_{i+4} \N v_{i+6} = \delta$, because the tile is of type~4. In the
former case, the vertex $i$ has type $2\alpha+\beta$ while in the latter
case, the vertex ${i+4}$ has type $2\alpha+\beta$. By comparing it with
the type of the vertex $i+2$, we have $\alpha=\gamma$ in either
case. This contradicts against Lemma~\ref{lem:6.2}. Finally, we have
$B\ne\alpha+3\delta$. Otherwise, by a similar argument, the length of
edges $\S v_{i+3}$, $v_{i+2} \N$ and $v_{i+6} \N$ are all $c$, the type
of the vertex $i+3$ is $2\gamma+\beta$, the types of the angles $\angle
\S v_{i+1} v_{i+2}$ and $\angle \S v_{i+5} v_{i+4}$ are both $\alpha$,
and the types of $\angle v_i v_{i+1} v_{i+2}$ and $v_{i+4} v_{i+5}
v_{i+6}$ are both $\beta$. Since we assumed $B=\alpha+3\delta$, the type
of $\angle v_{i+1} \S v_{i-1}$ or the type of $\angle v_{i+7} \S
v_{i+5}$ is $\delta$, from which $2\alpha+\beta$ follows. Therefore
$\alpha=\gamma$. This contradicts against Lemma~\ref{lem:6.2}.  Hence
the tile is  of type~2, and thus we can apply
Lemma~\ref{lem:u}. As the vertices $v_i$ and $v_{i+2}$ satisfy the
conditions~(ii) and (i) of Lemma~\ref{lem:u}, we have $\alpha\ne\delta$.
Then $A$ is $B$, because otherwise \eqref{claim:aho}
implies $\alpha=\delta$.  

The four designated inner angles around the vertex $F$ bijectively
correspond to those around the vertex $F+1$ via the automorphism $h$
given in Remark~\ref{rem:automorphism}.  The distinguished type, say
$\delta$, is aligned to exactly one of the four designated inner angles
around the vertex $F$, and so is to exactly one of the four designated
inner angles around the vertex $F+1$. Hence, there are $\binom{4}{2}=6$
pairs of non-$h$-equivalent inner angles and are 4 pairs of
$h$-equivalent inner angles. This ends the proof of Lemma~\ref{cor:homo}.
\qed\end{proof}

The following lemma forbids the ten patterns of Figure~\ref{fig:left_10possibilities}.
\begin{lemma}\label{lem:blade_impossible}
 A chart of a spherical tiling by congruent quadrangles of type~2 is
none of the three patterns (\casenine),(\casetwofour),(\theothercases) presented in Figure~\ref{fig:lem6.4} and is none of
 the three  (\casenine),(\casetwofour),(\theothercases) with $(\alpha,\beta)$ and $(\delta,\gamma)$ swapped in the
 same Figure.
\begin{figure}[ht]
\centering
\begin{tabular}{c c c}
\includegraphics[width=10cm,height=4cm]{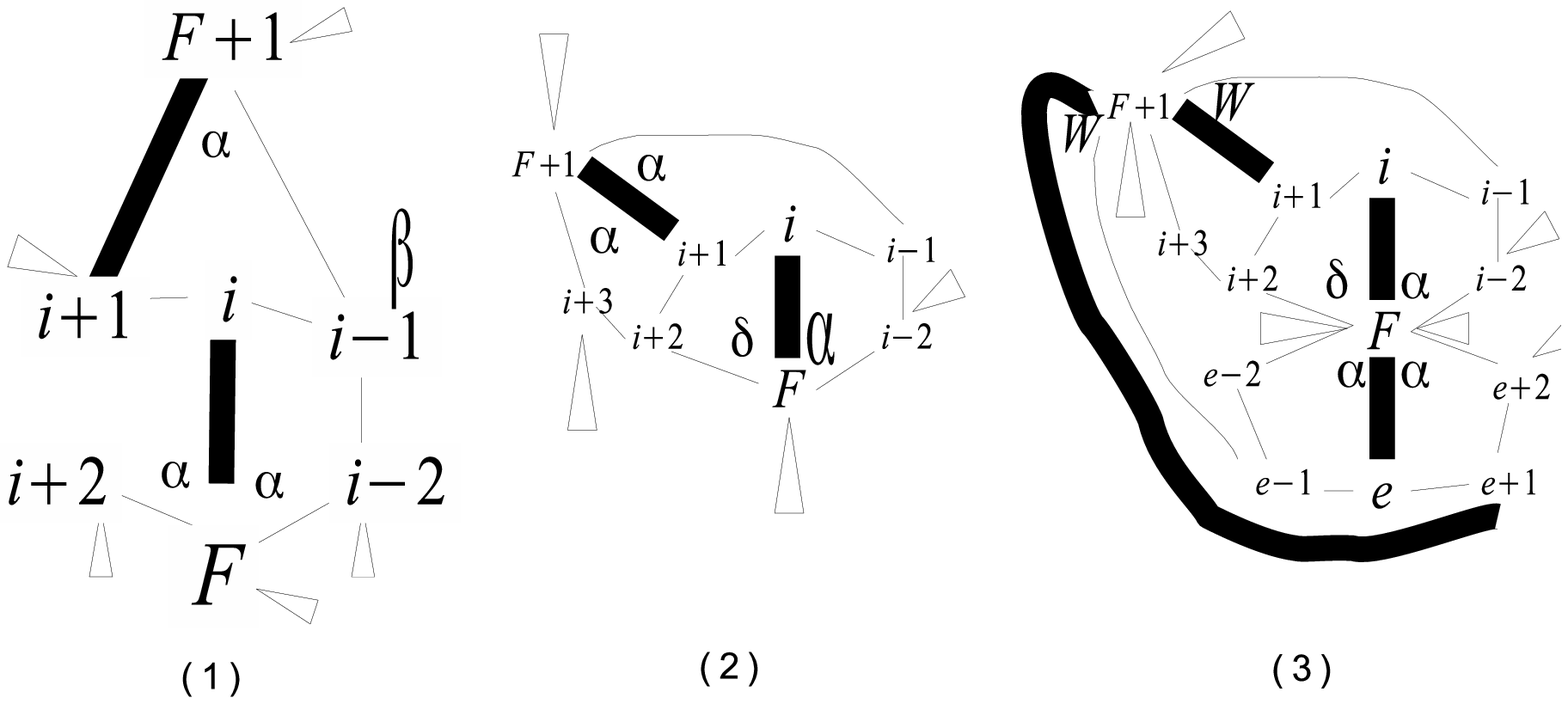}
\end{tabular} 
\caption{In the rightmost
 figure, the designated angle $\delta$ around the vertex $F$ is
 exchangeable with any designated angle $\alpha$ around the
 vertex.\label{fig:lem6.4} \label{fig:blade_impossible}}
\end{figure}
\end{lemma}

Theorem~\ref{thm:Forbidden}~\eqref{assert:forbidden_blade} is  proved through rejecting all the ten
charts in Figure~\ref{fig:left_10possibilities} as well as all the ten patterns with $\alpha$
and $\delta$ swapped, as follows:
Case~9 in Figure~\ref{fig:left_10possibilities} is impossible by the 
pattern~(\casenine) of Figure~\ref{fig:blade_impossible}, Case~2 by the
 pattern~(\casetwofour), Case~4 by applying the same 
pattern~(\casetwofour) with $(v_F, v_{F+1})$ being (the infinite point, the
central point), and the other cases by the pattern~(\theothercases).
All the ten patterns in Figure~\ref{fig:left_10possibilities} with $\alpha$
and $\delta$ swapped are rejected similarly.

\bigskip
The rest of this subsection is devoted to the proof of Lemma~\ref{lem:blade_impossible}.
(\casenine) We prove that the
pattern~(\casenine) of Figure~\ref{fig:blade_impossible} is impossible
 for a type-2 chart, by  establishing the following:
there do not exist two type-2 tiles $v_i v_{i+1} v_{F+1} v_{i-1}$ and
 $v_F v_i v_{i-1} v_{i-2}$ such that they match at the length-$a$ edge
 $v_i v_{i-1}$, (\blgh) the opposite edge $v_{i+1} v_{F+1}$ to the
 matching edge $v_i v_{i-1}$ in the first tile has length $b$, and an adjacent
 edge $v_i v_F$ to the matching edge $v_i v_{i-1}$ in the second tile
 does length $b$;
 (\alphaangle)
the two inner angles at an extreme point $v_F$ of the edge $v_i v_F$
and $\angle v_{i+1} v_{F+1} v_{i-1}$ have type $\alpha$; and
(\betaangle)
$\angle
 v_{i-2} v_{i-1}  v_{F+1} =\beta$. 

Assume some spherical tiling by congruent quadrangles of type~2
 satisfies all the three conditions.  By the condition~(\blgh) and the
 type of the first two angles of the condition~(\alphaangle), the two
 inner angles at the other extreme point $v_i$ of the edge $v_i v_F$,
 namely, $\angle v_F v_i v_{i+1}$ and $\angle v_F v_i v_{i-1}$, have
 type $\delta$. The two inner angles are adjacent to the inner angle
 $\angle v_{i+1} v_i v_{i-1}$ of type $\gamma$, because of the first
 tile $v_i v_{i+1} v_{F+1} v_{i-1}$ and because the last angle $\angle
 v_{i-1} v_{F+1} v_{i+1}$ of the condition~(\alphaangle) has type
 $\alpha$. Therefore $\gamma+2\delta=2$. On the other hand, by $v_{i+1}
 v_{F+1}= b$ and the condition~(\alphaangle), we have $\angle v_{F+1}
 v_{i-1} v_i =\beta$.  Because of the second tile $v_F v_i v_{i-1}
 v_{i-2}$ and because $\angle v_i v_F v_{i-2}=\alpha$ by the
 condition~(\alphaangle), we have $\angle v_i v_{i-1}
 v_{i-2}=\gamma$. By this and the condition~(\betaangle), we have
 $\gamma+2\beta=2$.  Therefore $\beta=\delta$, which contradicts against
 Lemma~\ref{lem:6.2}. By the argument above with $(\alpha,\beta)$ and
 $(\delta,\gamma)$ swapped, we have $\gamma=\alpha$, which also
 contradicts against Lemma~\ref{lem:6.2}.

Similarly, we can prove that the pattern~(\casenine) of Figure~\ref{fig:blade_impossible} with
$(\alpha,\beta)$ and $(\delta,\gamma)$ swapped is impossible for a
type-2 chart.

\bigskip (\casetwofour) We prove that the pattern~(\casetwofour) in
 Figure~\ref{fig:blade_impossible} is impossible for a type-2
 chart. Assume otherwise. Then there exist four tiles $v_F v_{i+2} v_{i+1} v_i$, $v_{i+2} v_{i+3}
 v_{F+1} v_{i+1}$, $v_{i-1} v_i v_{i+1} v_{F+1}$ and $v_i v_{i-1}
 v_{i-2} v_F$ such that

\medskip

\begin{enumerate}[$($i$)$]
\item
\label{threevalent} The two vertices $v_{i-1}$ and $v_{i+2}$ are 3-valent.

\item
\label{lghb} The two edges $v_i v_F$ and $v_{i+1} v_{F+1}$ have length $b$. 

\item \label{alphaangle} The angle $\angle v_{i-2} v_F v_i$, as well as
the two designated angles around the vertex $v_{F+1}$, $\angle v_{i+1} v_{F+1} v_{i-1}$, and
$\angle v_{i+3} v_{F+1} v_{i+1}$, have types $\alpha$.

\item 
\label{vFdelta}  $\angle v_i v_F v_{i+2}$ has type $\delta$.
\end{enumerate}

\medskip   By the conditions~\eqref{alphaangle} and
\eqref{vFdelta}, the vertex $i$ has type \begin{align}
 \alpha+\gamma+\delta=2\label{c1}
\end{align} 
while the vertex  $i+1$ has type
\begin{align}
\beta+2\delta=2\label{c2}.
\end{align}
By the conditions~\eqref{threevalent} and \eqref{lghb}, the lengths of the two edges $v_{i+2}
 v_{i+3}$, and  $v_{i+2} v_F$ are $a$ and  $\angle v_F v_{i+2} v_{i+3}$ has type
  $Y\in \{\beta,\gamma\}$.  So
 the vertex $i+2$ has type
\begin{align}
 Y+2\gamma = 2,\quad (Y\in \{\beta,\gamma\}.) \label{c3}
\end{align}
When $Y=\beta$, 
 \eqref{c2} implies $\gamma=\delta$, \eqref{c1} becomes
 $\alpha+2\delta=2$, and \eqref{c2} implies $\alpha=\beta$. This contradicts
 against Lemma~\ref{lune}.

When $Y=\gamma$, the equation~\eqref{c3} implies $\gamma=2/3$. By
 condition~\eqref{lghb}, the two edges $v_{i-1} v_{F+1}$ and $v_{i-1}
 v_{i-2}$ have length $a$. Hence
 the vertex $i-1$ is 3-valent by condition~\eqref{threevalent} and has
 type $2\beta+\gamma$ or $\beta+2\gamma$. In either case,
 $\beta=\gamma=2/3$. This contradicts against Lemma~\ref{lem:u}.

Similarly, we can prove that the pattern~(\casetwofour) with
$(\alpha,\beta)$ and $(\delta,\gamma)$ swapped is impossible for a
type-2 chart.
\bigskip

(\theothercases) We prove the chart~(\theothercases) of
 Figure~\ref{fig:blade_impossible} is impossible for a type-2
 chart. Assume otherwise. Because the types of $\angle v_{i+1} v_{F+1}
 v_{i-1}$ and $\angle v_{e+1} v_{F+1} v_{e-1}$ are equal and because the
 four edges $v_i v_F, v_{i+1} v_{F+1}, v_e v_F, v_{e+1}$ and $v_{F+1}$
 have length $b$, the type of $\angle v_{i-1} v_i v_{i+1}$ and that of
 $\angle v_{e-1} v_e v_{e+1}$ are equal. Since the sum of the types of
 four angles around the vertex $F$ $\angle v_{i-2} v_F v_i$, $\angle v_i
 v_F v_{i+2}$, $\angle v_{e-2} v_F v_e$, $\angle v_e v_F v_{e+2}$ is
 $3\alpha+\delta$ or $\alpha+3\delta$, we have $\alpha=\delta$.

 The vertex $v_{i-1}$ is 3-valent and incident to three edges of length
$a$.  There are two type-2 tiles $v_{i-1} v_i v_{i+1} v_{F+1}$ and $v_i
v_{i-1} v_{i-2} v_F$.  Moreover, the vertex $v_{i+1}$ is 3-valent and
incident to a length-$a$ edge and a length-$b$ edge. By
Lemma~\ref{lem:u}, $\alpha=\delta$. We have a contradiction.

Similarly, we can prove that the pattern~(\theothercases) with
$(\alpha,\beta)$ and $(\delta,\gamma)$ swapped is impossible for a
type-2 chart. This completes the proof of Lemma~\ref{lem:blade_impossible}.
\qed

\medskip
Hence, we have proved  Theorem~\ref{thm:Forbidden}~\eqref{assert:forbidden_blade}.\qed

\subsection{The proof of
  Theorem~\ref{thm:Forbidden}~\eqref{assert:forbidden_typhoon} 
\label{subsec:latter}}

\begin{lemma}\label{lem:typhoon_list} If a chart of a spherical tiling
 by congruent quadrangles is of
 the form
 Figure~\ref{fig:Forbidden}~(right), then the vertices
 $i-1$ and  $i+2$ have the same type, and the tile is of type~2. In
 this case, the list of such  charts constitutes ten patterns of
 Figure~\ref{fig:right_10possibilities} and those with
 $(\alpha,\beta)$ and $(\delta,\gamma)$ swapped.
\begin{figure}[ht]
\centering
\includegraphics[scale=0.6,width=9cm, height=9cm]{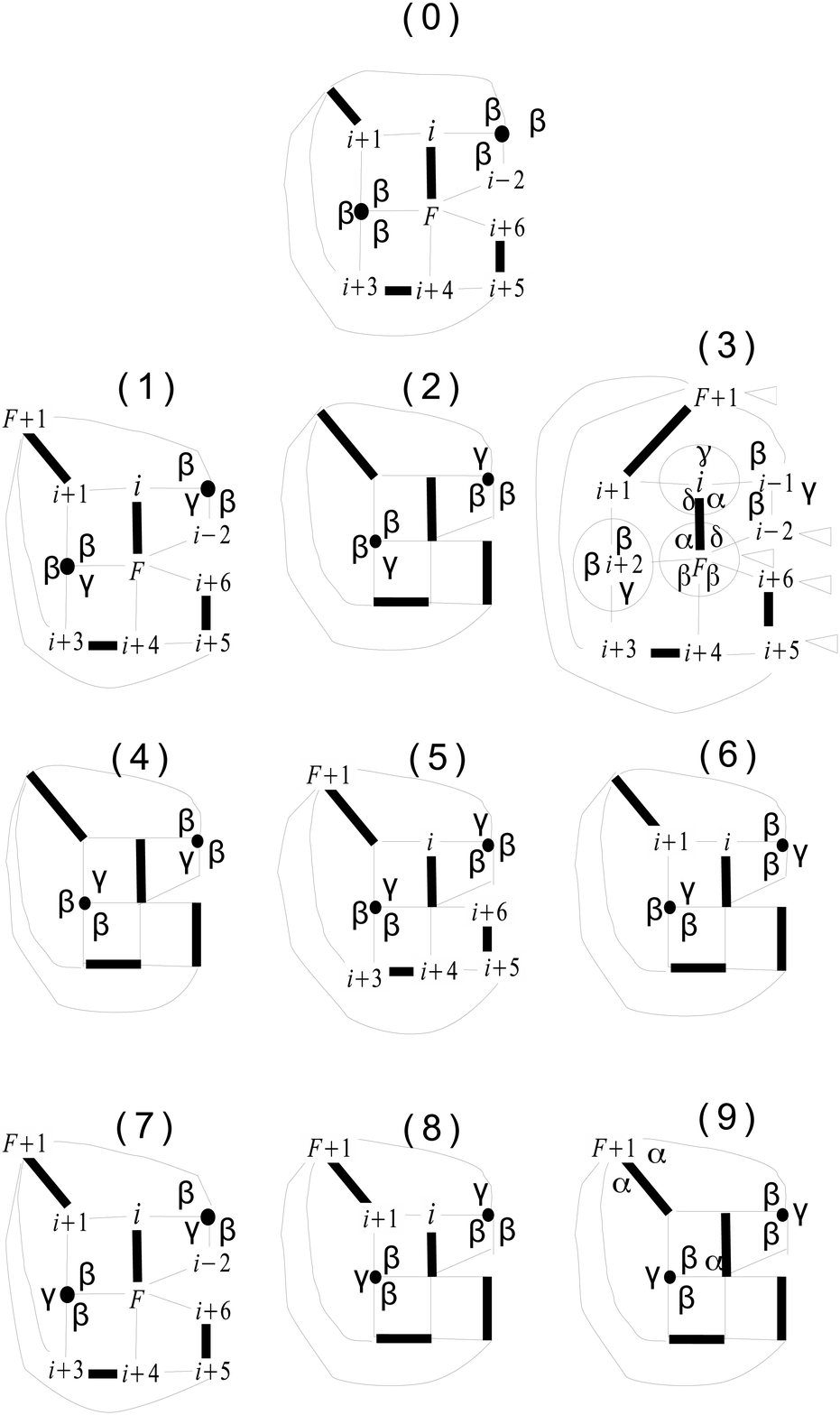}
\caption{The ten patterns in this figure and those with the designated
  $\beta$'s and the designated $\gamma$'s swapped forms the complete
  list of charts of the form Figure~\ref{fig:Forbidden}~(right). See
  Lemma~\ref{lem:typhoon_list}.  All charts are subject to common
  naming convention of the vertices.
  \label{fig:right_10possibilities}}
\end{figure} 
\end{lemma} 

\begin{proof}
We first verify that the tile is a quadrangle of type~2.
Assume otherwise. Then the tile is a quadrangle of type~4. If $\angle
 v_i v_F v_{i+2}$ of Figure~\ref{fig:Forbidden}~(right) has type $\delta$, then the edge
 $v_F v_{i+2}$ has length $c$, but is not incident to an edge of length
 $b$ in a quadrangle $v_F v_{i+4} v_{i+3} v_{i+2}$, which is a
 contradiction. Therefore
$\angle v_i v_F v_{i+2}$ has type $\alpha$, and thus $\angle v_F v_i
 v_{i+1}$ has type $\delta$. Because two edges of length $c$ are not adjacent,
$\angle v_F v_i v_{i-1}$ has type $\alpha$ and $\angle v_{i-1} v_i
 v_{i+1}$ does $\gamma$. 
On the other hand, the edge $v_{i+2} v_{i+3}$ of the quadrangle $v_{i+1}
 v_{F+1} v_{i+3} v_{i+2}$ has length $c$, since the edge is not adjacent to
 the $b$-length edge $v_{i+1} v_{F+1}$ of the quadrangle. Therefore the
 edge $v_{i+4} v_F$ has length $c$ and $\angle v_{i+2} v_F v_{i+4}$ has
 type $\gamma$. Hence the type of $v_F$ subsumes the type
 $\alpha+\gamma+\delta$ of $v_i$ and thus is greater than 2, which is a
 contradiction. Hence the tile is of type~2.

 The vertex $v_{i+2}$ is 3-valent and is incident to three
 length-$a$ edges. Furthermore, the vertex $v_i$ is 3-valent and is
 incident to a length-$a$ edge and to a length-$b$ edge. Hence
 Lemma~\ref{lem:u}  implies $\beta \ne \gamma$. 

However, since the vertex $v_{i-1}$ is 3-valent and is incident to three
 length-$a$ edges, 
if the vertices $v_{i-1}$ and $v_{i+2}$ have different types, then
 $\beta=\gamma=2/3$. This is a contradiction. Thus, the two vertices
 $v_{i-1}$ and $v_{i+2}$ have the same type.

By enumerating all the possibilities of the types of $v_{i-1}$ and
 $v_{i+2}$,  we have ten patterns of
Figure~\ref{fig:right_10possibilities} and the ten patterns with
$\beta$ and $\gamma$ swapped. \qed\end{proof}

To finish the proof of Theorem~\ref{thm:Forbidden}, we reject all ten
cases in Figure~\ref{fig:right_10possibilities}.  Case~0 is impossible
because the vertex $i$ has type $\alpha+\gamma+\delta$ which is a
subtype of the type of the vertex $F$.

\begin{lemma}\label{lem:case_three_forbidden_typhoon}If a spherical
 tiling by $F$ congruent quadrangles satisfies $F=8$ or the \emph{convexity}
 of the tile, then Case~3 of Figure~\ref{fig:right_10possibilities} is impossible,
\end{lemma}
\begin{proof} The angle $\angle v_{i+2} v_F v_{i+4}$ has type $\beta$
because $\angle v_{i+3} v_{i+2} v_F$ has type $\gamma$ and the edge
$v_{i+2} v_F$ has length $a$. Moreover the angle has the same type as
the angle $\angle v_{i+2} v_F v_{i+4}$ because otherwise, we have
$\alpha+\beta+\gamma+\delta\le 2$, which contradicts against
Lemma~\ref{lem:area}. We have $\alpha+2\beta+\delta\le2$ (the equality
holds when the number of faces is 8) and $\alpha+\gamma+\delta=2$
because the edge $v_i v_F$ has length $b$. Since $2\beta+\gamma=2$ from
Figure~\ref{fig:right_10possibilities}~(Case~3), we have $\gamma\ge1$ (the equality
holds when the number of faces is 8), which contradicts against the
convexity of the tile (against fact tiling is edge-to-edge, when the
number of faces is 8).
\qed \end{proof}

The other nine cases are rejected without using the tile's convexity
assumption, as follows:

Each of Case~0, 2, 4 is impossible because the vertex $i$ has type $\alpha+\gamma+\delta,
     \alpha+\beta+\delta,\alpha+\gamma+\delta$ which are subtypes of the
     type of the vertex $F$, in respective case.

Case~1 and Case~7 both contradict against Lemma~\ref{lem:blade_impossible}~(\casenine), and
Case~9 does against Lemma~\ref{lem:blade_impossible}~(\casenine) applied
to $v_{F+1}$.

In Case~6, the type of the vertex $i+1$ is $\alpha+\beta+\delta$. By
comparing the type $2\alpha+\gamma$ of the vertex $i$ and the type
$2\beta+\gamma$ of the vertex  $i-1$, we have 
 $\alpha=\beta$. By this, the type $2\beta+\gamma$ of the
 vertex $i-1$ is equal to $\alpha+\beta+\gamma$. So by comparing it with
 the type $\alpha+\beta+\delta$ of the vertex $i+1$, we have
 $\gamma=\delta$ as well as $\alpha=\beta$. Since Lemma~\ref{lem:typhoon_list} implies the tile is
 of type~2, we have a contradiction against Lemma~\ref{lune}. 

In Case~8, by comparing the type  $\alpha+\gamma+\delta$ of the vertex $i+1$ 
and the type  $\alpha+\beta+\delta$ of the vertex $i$, we have
$\beta=\gamma$, which contradicts against Lemma~\ref{lem:u} because
Lemma~\ref{lem:typhoon_list} implies the tile is of type~2.

In Case~5, the type of vertex $i$ is $2\alpha+\beta=2$.
the type of  $\angle v_{i+4} v_F  v_{i+6}$ is $\beta$ or $\gamma$. First
consider the case $\angle v_{i+4} v_F v_{i+6}$ is $\beta$. 

$\angle v_{F+1} v_{i+3} v_{i+4}$ has type $\delta$. Otherwise it has
$\alpha$. By comparing the types of the vertex $i$ and the vertex $i+3$,
we have $\beta=\gamma$, which contradicts against Lemma~\ref{lem:u}
because Lemma~\ref{lem:typhoon_list} implies the tile is of type~2.

We see that $\angle v_{F+1} v_{i+5} v_{i+4}$ has type $\beta$ and 
 $\angle v_{i+6} v_{i+5} v_{i+4}$ does $\delta$. So  $\angle
v_{i+6} v_{i+5} v_{F+1}$ does $\alpha$. Otherwise $\gamma$ must be
0 because the type of vertex $v_F$ is less than or equal to
$\gamma+\beta+2\delta$.

Then $\beta=\gamma$ by comparing the types of vertices $i+4$ and $i+5$.
Because the tile of type~2 by Lemma~\ref{lem:typhoon_list}, we have a
contradiction against Lemma~\ref{lem:u}. So 
$\angle v_{i+4} v_F v_{i+6}$ does not have type $\beta$ in Case~5.

Next consider the case $\angle v_{i+4} v_F v_{i+6}$ has type $\gamma$ in
Case~5.  Then $\angle v_{i+5} v_{i+4} v_F$ has type $\beta$ and $\angle
v_F v_{i+4} v_{i+3}$ does $\delta$, and $\angle v_{i+3} v_{i+4} v_{i+5}$
has type $\alpha$ or $\delta$. By comparing the types of $i$ and $i+4$,
we have $\alpha=\delta$, which contradicts against Lemma~\ref{lem:u}
because the tile is of type~2 by Lemma~\ref{lem:typhoon_list}.  

All the arguments are still valid even if we swap $(\alpha,\beta)$ and
$(\delta,\gamma)$. So this completes the proof of
Theorem~\ref{thm:Forbidden}.  \qed\bigskip

\section{Discussion\label{sec:future_work}}

By mostly combinatorial argument, we have derived Theorem~\ref{thm:theo}, a
 spherical geometric theorem.

Every spherical tiling by congruent \emph{convex} quadrangles over a
pseudo-double wheel is isohedral, by Theorem~\ref{thm:theo} for type-2 and
type-4 tiles, by \cite{sakano13:_class_of_spher_tilin_by_kdr} for a tile
being a kite or a rhombus. But some by congruent \emph{concave}
quadrangles over a pseudo-double wheel is not
isohedral~(Theorem~\ref{thm:k}). So we conjecture that the ``inverse''
of Gr\"unbaum-Shephard's result holds if the tile is convex. That is,
\begin{conjecture}\label{conj:inverse}For every normal spherical monohedral
tiling topologically a Platonic solid, an Archimedean dual, an $n$-gonal
bipyramid, or an $n$-gonal trapezohedron~($n\ge3$), if the tile is
\emph{convex}, then the tiling is isohedral.\end{conjecture}

Because of Theorem~\ref{thm:inverse_triangle}, the conjecture is
completely settled down if we work it out for the non-triangle-faced
Platonic solids and such Archimedean duals. The classification of
spherical tiling by twelve congruent pentagons~\cite{MR3022611} seems useful for the solution.
We hope the partial
mechanization described in Remark~\ref{remark:mechanize} followed by
trigonometric arguments is a feasible strategy.

\subsection{Classification of spherical monohedral quadrangular tilings}
The classification of spherical monohedral quadrangular tilings solves
``spherical Hilbert's eighteenth problem'': Enumerate spherical
anisohedral triangles, anisohedral quadrangles and anisohedral
pentagons. By an \emph{anisohedral} tile, we mean a tile that admits a
spherical monohedral quadrangular tiling but not a spherical
\emph{isohedral} tiling. There are two infinite series of spherical
anisohedral triangles but none of the other triangles, kites, darts, rhombi are
anisohedral~\cite{sakano13:_class_of_spher_tilin_by_kdr}.

The only spherical tilings by congruent quadrangles of type~2 or 4 over
pseudo-double wheels are $\P_F$ or $\A$, \emph{if} we can drop the
tile's convexity assumption from Theorem~\ref{thm:Forbidden},
specifically from Lemma~\ref{lem:2a2d} and
Lemma~\ref{lem:case_three_forbidden_typhoon}.  If there is, however,
\emph{another} spherical monohedral tiling by congruent \emph{concave}
quadrangles of type~2 or type~4 with the map being a pseudo-double
wheel, then such a tiling necessarily has a meridian edge of length $b$
by Theorem~\ref{thm:kouho2:one}.

Brinkmann's group observed that
the graphs of  spherical monohedral quadrangular tilings can be  enumerating efficiently
if the number of distinct degrees of vertex is known~(See Table~\ref{tbl:degrees}).
\begin{table}[htb]\small\centering
\begin{tabular}{| r || r | r | r | r | r | r |}
\hline
\backslashbox{$F$}{$\Delta$} & 1  &   2 &      3 &      4 &    5 & 6 \\
\hline
   6 & 1  \\
   8 & 0  &   1 \\
  10 & 0  &   3 \\
  12 & 0  &   7 &       5 \\
  14 & 0  &  11 &      43 &       10 \\
  16 & 0  &  13 &     298 &      199 \\
  18 & 0  &  46 &    1937 &     2981 &      182 \\
  20 & 0  &  33 &   13792 &    38715 &     6242 \\
  22 & 0  & 103 &  100691 &   474123 &   141059 &     631 \\
  24 & 0  & 224 &  758959 &  5596936 &  2658188 &   48095 \\
  26 & 0  & 433 & 5808034 & 64603662 & 45200498 & 1885445 \\
\hline
\end{tabular}
\caption{The number $p(F,\Delta)$ of 2-connected simple spherical
quadrangulations such that the minimum degree is three, the number $F$
of faces is even, and the number of distinct degrees is $\Delta$.  Here $p(F,\Delta)$
is the number of isomorphism classes if orientation-reversing
(reflectional) isomorphisms are permitted. By \cite[Table~2]{MR2186681}, $p(F,\Delta)=0$ if
 $6\le F\le 26$  and  $\Delta\ge7$.
By the courtesy of  Van
Cleemput.\label{tbl:degrees}}\end{table}
 Brinkmann conjectured that there are at most three distinct
vertex types in a chart of any spherical monohedral quadrangular tiling,
motivated by  linear algebraic~(or linear programming) argument that independent
equations of $\alpha,\beta,\gamma,\delta$ arise from the distinct vertex
types. Indeed, the conjecture holds for every
spherical monohedral tiling such that the tile is a kite, a dart, or a
rhombus, according to the classification~\cite{sakano13:_class_of_spher_tilin_by_kdr}. 

Even if we can enumerate all graphs of spherical monohedral
type-2/type-4 quadrangular tiling by computers, to show that spherical
tilings by congruent \emph{concave} quadrangles exist requires difficult
trigonometric arguments because there is a spherical concave quadrangle
$Q$ such that there are continuously many spherical non-congruent
quadrangles $Q'$ with the same cyclic list of inner angles of $Q$,
according to \cite{akama13:_spher_tilin_by_congr_quadr_ii}.

We propose to classify 
\begin{enumerate}\item
the spherical tilings by congruent
tilings such that the tile is possibly concave and the graphs are the spherical monohedral
(kite/dart/rhombic)-faced
tilings~\cite{sakano13:_class_of_spher_tilin_by_kdr}; and

\item the spherical tilings by congruent \emph{convex} quadrangles of
		      type~2. 

Computer experiments using \cite{plantri,MR2186681} show that for sufficiently high number of faces
more than 40\% of spherical quadrangulations can already be excluding as
                                  the map of a spherical tiling by congruent convex
                                  type-2 quadrangles by checking for
		      some small forbidden substructures~\cite{akama13:_spher_tilin_by_congr_quadr}.
\end{enumerate}

\end{document}